\documentclass[10pt]{amsart}
\usepackage{amssymb,latexsym}
\usepackage{mathrsfs}
\usepackage{amssymb}
\usepackage{amsfonts}
\usepackage{graphicx}
\usepackage{color}

\DeclareMathOperator{\trace}{Tr}
\DeclareMathOperator{\divergence}{div}

\newtheorem{theorem}{Theorem}[section]
\newtheorem{definition}[theorem]{Definition}
\newtheorem{lemma}[theorem]{Lemma}
\newtheorem{proposition}[theorem]{Proposition}
\newtheorem{remark}[theorem]{Remark}

\newtheorem{claim}{Claim}

\numberwithin{equation}{section}

\makeatletter
\def\blfootnote{\xdef\@thefnmark{}\@footnotetext}
\makeatother

\begin{document}
\vspace{1 true cm}

\title[]{Asymptotic behavior of Palais-Smale sequences\\ associated with fractional Yamabe type equations}%

\author[]{Yi Fang}%

\address{Yi Fang, School of Mathematical Sciences, University of Science and Technology of China, Hefei, Anhui 230026, P.R.China}
\email{fangyi1915@gmail.com}

\author[]{Maria del  Mar Gonzalez}%

\address{Maria del Mar Gonzalez, Departament de Matem\`tica Aplicada, Universitat Polit\`ecnica de Catalunya, Av. Diagonal 647, Barcelona 08028, Spain}
\email{mar.gonzalez@upc.edu}
   
\maketitle

\blfootnote{Y. Fang is fully supported by China Scholarship Council(CSC) for visiting University of California, Santa Cruz. M.d.M. Gonz\'alez is supported by Spain Government project MTM2011-27739-C04-01 and GenCat 2009SGR345.}

{\bf Abstract.}\ \ In this paper, we analyze the asymptotic behavior of Palais-Smale sequences associated to fractional Yamabe type equations on an asymptotically hyperbolic Riemannian manifold. We prove that Palais-Smale sequences can be decomposed into the solution of the limit equation plus a finite number of bubbles, which are the rescaling of the fundamental solution for the fractional Yamabe equation on Euclidean space. We also verify the non-interfering fact for multi-bubbles.


\vspace{0.5 true cm}
\section{Introduction and statement of results}

Let $\Omega$ be a smooth bounded domain in $\mathbb{R}^n$, $n\geq3$. Fix a constant $\lambda$, and consider the Dirichlet boundary value problem of the elliptic PDE
\begin{equation}\label{Euclidean}
    \left\{
\begin{split}
  -\Delta u-\lambda u& =u|u|^{\frac{4}{n-2}} &\hbox{in }\ \Omega, \\
     u&=0 &\hbox{on} \ \partial\Omega.
 \end{split}
\right.
\end{equation}
The associated variational functional of the equation \eqref{Euclidean} in the Sobolev space $W^{1,2}_0(\Omega)$ is
$$
E(u)=\frac{1}{2}\int_\Omega(|\nabla u|^2-\lambda u^2)\,dx-\frac{n-2}{2n}\int_\Omega|u|^{\frac{2n}{n-2}}\,dx.
$$
Suppose that the sequence $\{u_\alpha\}_{\alpha\in \mathbb{N}}\subset W^{1,2}_0(\Omega)$ satisfies the Palais-Smale condition, i.e.
$$\{E(u_\alpha)\}_{\alpha\in \mathbb{N}}\text{ is uniformly bounded and }DE(u_\alpha)\rightarrow0, \mbox{ strongly in }(W^{1,2}_0(\Omega))',$$
as $\alpha\rightarrow+\infty$, where  $(W^{1,2}_0(\Omega))'$ is the dual space of $W^{1,2}_0(\Omega)$. In an elegant paper \cite{S}, M. Struwe considered the asymptotic behavior of  $\{u_\alpha\}_{\alpha\in \mathbb{N}}$. In fact, in the $W^{1,2}_0(\Omega)$ norm, $u_\alpha$ can be approximated by the solution to \eqref{Euclidean} plus a finite number of bubbles, which are the rescaling of the non-trivial entire solution of
$$
-\Delta u=u|u|^{\frac{4}{n-2}} \ \ \hbox{in} \ \ \mathbb{R}^n \ \ \hbox{and}  \ \ u(x)\rightarrow0 \ \ \hbox{as} \ \ |x|\rightarrow+\infty.
$$

One may pose the analogous problem on a manifold. Let $(M^n,g)$ be a smooth compact Riemannian manifold without boundary. Consider a sequence of  elliptic PDEs like
\begin{equation}\label{Ealpha}
-\Delta_g u+h_\alpha u=u^{\frac{n+2}{n-2}},\tag{$E_\alpha$}
\end{equation}
where $\alpha\in \mathbb{N}$ and $\Delta_g$ denotes the Laplace-Beltrami operator of the metric $g$. Assume that $h_\alpha$ satisfies that there exists $C>0$ with $|h_\alpha(x)|\leq C$ for any $\alpha$ and any $x\in M$; also $h_\alpha\rightarrow h_\infty$ in $L^2(M)$ as $\alpha\rightarrow +\infty$. The limit equation is denoted by
\begin{equation}\label{Einfty}
-\Delta_g u+h_\infty u=u^{\frac{n+2}{n-2}}\tag{$E_\infty$}.
\end{equation}
The related variational functional for \eqref{Ealpha} is
$$
E^\alpha_g(u)=\frac{1}{2}\int_M|\nabla u|^2_gdv_g+\frac{1}{2}\int_Mh_\alpha u^2dv_g-\frac{n-2}{2n}\int_M |u|^{\frac{2n}{n-2}}dv_g.
$$
Suppose that $\{u_\alpha\geq 0\}_{\alpha\in \mathbb{N}}\subset W^{1,2}(M)$ also satisfies the Palais-Smale condition. O. Druet, E. Hebey and F. Robert \cite{D-H-R} proved that, in the $W^{1,2}(M)$-sense, $u_\alpha$ can be decomposed into the solution of $(E_\infty)$ plus a finite number of bubbles, which are the rescaling of the non-trivial solution of
$$
-\Delta u=u^{\frac{n+2}{n-2}} \ \ \hbox{in} \ \ \mathbb{R}^n.
$$

Next, let $(M^n,g)$ be a compact Riemannian manifold with boundary $\partial M$. Recently, S. Almaraz \cite{A} considered the following sequence of equations with nonlinear boundary value condition
\begin{equation}\label{Riemannian}
  \left\{
  \begin{split}
   -\Delta_g u&=0 &\quad\hbox{in } \ M,\\
    -\frac{\partial}{\partial\eta_g}u+h_\alpha u&=u^{\frac{n}{n-2}}  &\quad\hbox{on} \ \partial M,
 \end{split}
\right.
\end{equation}
where $\alpha\in \mathbb{N}$ and $\eta_g$ is the inward unit normal vector to $\partial M$. The associated energy functional for equation \eqref{Riemannian} is
$$
{\bar{E}}^\alpha_g(u)=\frac{1}{2}\int_M|\nabla u|^2_gdv_g+\frac{1}{2}\int_{\partial M}h_\alpha u^2d\sigma_g-\frac{n-2}{2(n-1)}\int_{\partial M}|u|^{\frac{2(n-1)}{n-2}}d\sigma_g,
$$
for $u\in H^1(M):=\{u|\nabla u\in L^2(M), u\in L^2(\partial M)\}$. Here $dv_g$ and $d\sigma_g$ are the volume forms of  $M$ and $\partial M$, respectively. He also showed that a nonnegative Palais-Smale sequence $\{u_\alpha\}_{\alpha\in \mathbb{N}}$ of $\{{\bar{E}}^\alpha_g\}_{\alpha\in\mathbb{N}}$ converges, in the $H^1(M)$-sense, to a solution of the limit equation (the equation replacing  $h_\alpha$ by $h_\infty$ in \eqref{Riemannian}) plus a finite number of bubbles.

Motivated by these facts and the original study of the fractional Yamabe problem by M.d.M. Gonz\'alez and J. Qing \cite{G-Q}, in this paper we shall be interested in the asymptotic behavior of nonnegative Palais-Smale sequences associated with the fractional Yamabe equation on an asymptotically hyperbolic Riemannian manifold.

Let ($X^{n+1},g^+$), $n\geq3$, be a smooth Riemannian manifold with smooth boundary $\partial X^{n+1}=M^n$.  A function $\rho_*$ is called a defining function of the boundary $ M^n$ in $X^{n+1}$ if it satisfies
$$\rho_*>0  \ \ \hbox{in} \ \ X^{n+1}, \ \ \rho_*=0  \ \ \hbox{on} \ \ M^n, \ \ d\rho_*\neq0  \ \ \hbox{on} \ \ M^n. \ \ $$
We say that a metric $g^+$ is conformally compact if there exists a defining function $\rho_*$ such that $(X^{n+1},\overline{g}_*)$ is compact for $\overline{g}_*=\rho_*^2g^+$. This induces a conformal class of metrics $\hat{h}=\overline{g}_*|_{M^n}$ when defining functions vary. The conformal manifold $(M^n,[\hat{h}])$ is called the conformal infinity of $(X^{n+1},g^+)$. A metric $g^+$ is said to be asymptotically hyperbolic if it is conformally compact and the sectional curvature approaches $-1$ at infinity. It is easy to check then that $|d\rho_*|^2_{\overline{g}_*}=1$ on $M^n$.

  Using the meromorphic family of scattering operators $S(s)$ introduced by C.R. Graham and M. Zworski \cite{G-Z}, we will define the so-called fractional order scalar curvature. Given an asymptotically hyperbolic Riemannian manifold $(X^{n+1},g^+)$ and a representative $\hat{h}$ of the conformal infinity $(M^n,[\hat{h}])$, there is a unique geodesic defining function $\rho_*$ such that, in $M^n\times(0,\delta)$ in $X^{n+1}$, for small $\delta$, $g^+$ has the normal form
$$
g^+=\rho_*^{-2}(d\rho_*^2+h_{\rho_*})
$$
where $h_{\rho_*}$ is a one parameter family of metric on $M^n$ such that
$$
h_{\rho_*}=\hat{h}+h^{(1)}\rho_*+O(\rho_*^2).
$$
It is well-known  \cite{G-Z} that, given $f\in \mathbb{C}^\infty(M^n)$, and $s\in\mathbb C$, $Re(s)>n/2$ and $s(n-s)$ is not an $L^2$ eigenvalue for $-\Delta_{g^+}$, then the generalized eigenvalue problem
\begin{equation}\label{eigenvalue1}
-\Delta_{g^+}\tilde u-s(n-s)\tilde u=0 \ \ \ \hbox{in}\ X^{n+1}
\end{equation}
has a solution of the form
$$
\tilde u=F(\rho_*)^{n-s}+G(\rho_*)^s, \ \ F,G\in\mathcal{C}^{\infty}(\overline{X}^{n+1}), \ \ F|_{\rho_*=0}=f.
$$
The scattering operator on $M^n$ is then defined as
$$
S(s)f=G|_{M^n}.
$$
Now we consider the normalized scattering operators
$$
P_\gamma[g^+,\hat{h}]=d_\gamma S\left(\frac{n}{2}+\gamma\right),  \ \ \ \ \ d_\gamma=2^{2\gamma}\frac{\Gamma(\gamma)}{\Gamma(-\gamma)}.
$$
Note  $P_\gamma[g^+,\hat{h}]$ is a pseudo-differential operator whose principal symbol is equal to the one of $(-\Delta_{\hat{h}})^\gamma$. Moreover, $P_\gamma[g^+,\hat{h}]$ is conformally covariant, i.e. for any $\varphi, w\in \mathcal C^\infty(\overline{X^{n+1}})$ and $w>0$, it holds
\begin{equation}\label{fractional covariant}
P_\gamma[g^+,w^{\frac{4}{n-2\gamma}}\hat{h}](\varphi)=w^{-\frac{n+2\gamma}{n-2\gamma}}P_\gamma[g^+,\hat{h}](w\varphi).
\end{equation}
Thus we shall call $P_\gamma[g^+,\hat{h}]$ the conformal fractional Laplacian for any $\gamma\in(0,n/2)$ such that $n^2/4-\gamma^2$ is not an $L^2$ eigenvalue for $-\Delta_{g^+}$.

The fractional scalar curvature associated to the operator $P_\gamma[g^+,\hat{h}]$ is defined as
$$
Q^{\hat{h}}_\gamma=P_\gamma[g^+,\hat{h}](1).
$$
The scattering operator has a pole at the integer values $\gamma$. However, in such cases the residue may be calculated and, in particular, when $g^+$ is Poincar\'{e}-Einstein metric, for $\gamma=1$ we have
$$
P_1[g^+,\hat{h}]=-\Delta_{\hat{h}}+\frac{n-2}{4(n-1)}R_{\hat{h}}
$$
is exactly the so-called conformal Laplacian, and
$$
Q^{\hat{h}}_1=\frac{n-2}{4(n-1)}R_{\hat{h}}.
$$
Here $R_{\hat{h}}$ is the scalar curvature of the metric $\hat{h}$.

 For $\gamma=2$,  $P_2[g^+,{\hat{h}}]$ is precisely the Paneitz operator and its associated curvature is known as $Q$-curvature \cite{P}.
 In general, $P_k[g^+,\hat{h}]$ for $k\in\mathbb{N}$ are precisely the conformal powers of the Laplacian studied in \cite{G-J-M-S}.

We consider the conformal change $\hat{h}_w=w^{\frac{4}{n-2\gamma}}\hat{h}$ for some $w>0$, then by \eqref{fractional covariant}, we have
$$
P_\gamma[g^+,\hat{h}](w)=Q^{\hat{h}_w}_\gamma w^{\frac{n+2\gamma}{n-2\gamma}}  \ \ \hbox{in}   \  (M^n,\hat{h}).
$$
If for this conformal change $Q^{\hat{h}_w}_\gamma$ is a constant $C_\gamma$ on $M^n$, this problem reduces to
\begin{equation}\label{equation0}
P_\gamma[g^+,\hat{h}](w)=C_\gamma w^{\frac{n+2\gamma}{n-2\gamma}}  \ \ \hbox{in}   \  (M^n,\hat{h}),
\end{equation}
which is the so-called the fractional Yamabe equation or the $\gamma$-Yamabe equation studied in \cite{G-Q}.

From now on, we always suppose that $\gamma\in(0,1)$ throughout the paper, and such that $n^2/4-\gamma^2$ is not an $L^2$ eigenvalue for $-\Delta_{g^+}$.\\

It is well known that the above fractional Yamabe equation may be rewritten as a degenerate elliptic Dirichlet-to-Neumann boundary problem. For that, we first recall some results obtained by S.A. Chang and M.d.M. Gonz\'alez in \cite{C-G}. Suppose that $u^*$ solves
\begin{equation}\label{eigenvalue2}
    \left\{
  \begin{split}
   -\Delta_{g^+} u^*-s(n-s)u^*=0 &\quad\hbox{in } \ X^{n+1},\\
   \lim_{\rho_*\to 0}\rho^{s-n}_* u^*=1  &\quad\hbox{on} \ M^n.
 \end{split}
\right.
\end{equation}

\begin{proposition}\label{special defining function proposition}\cite{C-G,G-Q}
Let $f\in\mathcal C^\infty(M)$. Assume that $\tilde u,u^*$ are solutions to \eqref{eigenvalue1} and \eqref{eigenvalue2}, respectively. Then $\rho=(u^*)^{1/(n-s)}$ is a geodesic defining function. Moreover, $u=\tilde u/u^*=\rho^{s-n}\tilde u$ solves
\begin{equation}
    \left\{
  \begin{split}
   -\divergence(\rho^{1-2\gamma}\nabla u)=0 &\quad\hbox{in } \ X^{n+1},\\
   u=f  &\quad\hbox{on} \ M^n,
 \end{split}
\right.
\end{equation}
with respect to the metric $g=\rho^2g^+$ and $u$ is the unique minimizer of the energy functional
$$
I(v)=\int_{X^{n+1}}\rho^{1-2\gamma}|\nabla v|^2_gdv_g
$$
among all the extensions $v\in W^{1,2}(X^{n+1},\rho^{1-2\gamma})$ (see Definition \ref{weighted sobolev space}) satisfying $v|_{M^n}=f$. Moreover,
$$
\rho=\rho_*\left(1+\frac{Q^{\hat{h}}_\gamma}{(n-s)d_\gamma}\rho^{2\gamma}_*
+O(\rho^2_*)\right)
$$
near the conformal infinity and
$$
P_\gamma[g^+,\hat{h}](f)=-d_\gamma^*\lim_{\rho\rightarrow0}\rho^{1-2\gamma}\partial_\rho u+Q^{\hat{h}}_\gamma f, \ \  d^*_\gamma=-\frac{d_\gamma}{2\gamma}>0,
$$
provided that $\trace_{\hat{h}}h^{(1)}=0$ when $\gamma\in(1/2,1)$. Here $g|_{M^n}=\hat{h}$, and has asymptotic expansion
$$
g=d\rho^2[1+O(\rho^{2\gamma})]+\hat{h}[1+O(\rho^{2\gamma})].
$$
\end{proposition}
\vskip 0.1in

We fix $\gamma\in(0,1)$. By Proposition \ref{special defining function proposition}, one can rewrite the Yamabe equation \eqref{equation0} into the following problem:
\begin{equation}\label{Dirichlet-to-Neumann}
\left\{
\begin{split}
-\divergence(\rho^{1-2\gamma}\nabla u)=0   &\quad\hbox{in} \  (X^{n+1},g),\\
u=w & \quad\hbox{on} \  (M^n,\hat{h}),\\
-d_\gamma^*\lim_{\rho\rightarrow0}\rho^{1-2\gamma}\partial_\rho u+Q^{\hat{h}}_\gamma w=C_\gamma w^{\frac{n+2\gamma}{n-2\gamma}}  & \quad\hbox{on}  \ (M^n,\hat{h}).\\
\end{split}\right.
\end{equation}
In this paper we consider the positive curvature case $C_\gamma>0$. Without loss of generality, we assume $C_\gamma=d_\gamma^*$.

In the particular case $\gamma=1/2$, one may check that \eqref{Dirichlet-to-Neumann} reduces to \eqref{Riemannian}, which was considered in \cite{A}. The main difficulty we encounter here is the presence of the weight that makes the extension equation only degenerate elliptic.\\

Next, we introduce the so-called $\gamma$-Yamabe constant (c.f. \cite{G-Q}). For the defining function $\rho$ mentioned above, we set
$$
I_\gamma[u,g]=\frac{d^*_\gamma\int_X\rho^{1-2\gamma}|\nabla u|^2_g\,dv_g+\int_MQ^{\hat{h}}_\gamma u^2\,d\sigma_{\hat{h}}}
{\left(\int_M|u|^{2^*}\,d\sigma_{\hat{h}}\right)^{\frac{2}{2^*}}},
$$
then the $\gamma$-Yamabe constant is defined as
\begin{equation}\label{Yamabe-constant}
\Lambda_\gamma(M,[\hat{h}])=\inf\{I_\gamma[u,g]: u\in W^{1,2}(X,\rho^{1-2\gamma})\}.
\end{equation}
It was shown in \cite{G-Q} that in the positive curvature case $C_\gamma>0$ we must have
$\Lambda_\gamma(M,[\hat{h}])>0$.

Now we take a perturbation of the linear term $Q^{\hat{h}}_\gamma w$ to a general $-d_\gamma^*Q^\gamma_\alpha w$,
where $Q^\gamma_\alpha\in \mathcal C^\infty(M^n)$, $\alpha\in \mathbb{N}$.
Suppose that for any $\alpha\in\mathbb{N}$ and any $x\in M^n$, there exists a constant $C>0$ such that $|Q^\gamma_\alpha(x)|\leq C$. And we also assume that $Q^\gamma_\alpha\rightarrow Q^\gamma_\infty$ in $L^2(M^n,\hat{h})$ as $\alpha\rightarrow+\infty$. We will consider a family of equations
\begin{equation}\label{divergence}
\left\{
\begin{split}
-\divergence(\rho^{1-2\gamma}\nabla u)=0   &\quad\hbox{in} \  (X^{n+1},g),\\
u=w & \quad\hbox{on} \ (M^n,\hat{h}),\\
-\lim_{\rho\rightarrow0}\rho^{1-2\gamma}\partial_\rho u+Q^\gamma_\alpha w= w^{\frac{n+2\gamma}{n-2\gamma}}  & \quad\hbox{on}  \ (M^n,\hat{h}).\\
\end{split}\right.
\end{equation}
The associated variational functional to \eqref{divergence} is
\begin{equation}\label{funcional}I_g^{\gamma,\alpha}(u)
=\frac{1}{2}\int_{X^{n+1}}\rho^{1-2\gamma}|\nabla u|^2_g\,dv_g
+\frac{1}{2}\int_{M^n}Q^\gamma_\alpha u^2\,d\sigma_{\hat{h}}
-\frac{n-2\gamma}{2n}\int_{M^n}|u|^{\frac{2n}{n-2\gamma}}d\sigma_{\hat{h}}.
\end{equation}

\vskip 0.1in
Hyperbolic space $(\mathbb{H}^{n+1}, g_{\mathbb{H}})$ is the first example of a conformally compact Einstein manifold. As   $(\mathbb{H}^{n+1}, g_{\mathbb{H}})$ can be characterized as the upper half-space $\mathbb{R}^{n+1}_+$ endowed with metric $g^+=y^{-2}(|dx|^2+dy^2)$, where $x\in\mathbb{R}^n$, $y\in\mathbb{R}_+$, then the Dirichlet-to-Neumann problem \eqref{Dirichlet-to-Neumann}  reduces to
\begin{equation}\label{Euclidean divergence}
\left\{
\begin{split}
-\divergence(y^{1-2\gamma}\nabla u)=0   &\quad\hbox{in} \  (\mathbb{R}^{n+1}_+,|dx|^2+dy^2),\\
u=w & \quad\hbox{on} \  (\mathbb{R}^n,|dx|^2),\\
-\lim_{y\rightarrow0}y^{1-2\gamma}\partial_{y}u=w^{\frac{n+2\gamma}{n-2\gamma}}  & \quad\hbox{on}  \ (\mathbb{R}^n,|dx|^2).\\
\end{split}\right.
\end{equation}
And the variational functional to \eqref{Euclidean divergence} is defined as
$$
\tilde{E}(u)=\frac{1}{2}\int_{\mathbb{R}^{n+1}_+}y^{1-2\gamma}|\nabla u(x,y)|^2dxdy
-\frac{n-2\gamma}{2n}\int_{\mathbb{R}^n}|u(x,0)|^{\frac{2n}{n-2\gamma}}dx.
$$
Up to multiplicative constants, the only solution to problem \eqref{Euclidean divergence} is given by the standard
$$
w(x)=w^\lambda_a(x)=\left(\frac{\lambda}{|x-a|^2+\lambda^2}\right)^{\frac{n-2\gamma}{2}}
$$
for some $a\in\mathbb{R}^n$ and $\lambda>0$ (c.f. \cite{G-Q},\cite{J-L-X}). By  L. Caffarelli and L. Silvestre's Poisson formula \cite{C-S}, the corresponding  extension can be expressed as
\begin{align}\label{Poissonformula}
U^\lambda_a(x,y)
=\int_{\mathbb{R}^n}\frac{y^{2\gamma}}{(|x-\xi|^2+y^2)^{(n+2\gamma)/2}}\,w^\lambda_a(\xi)\,d\xi.
\end{align}
Here $U^\lambda_a$ is called a ``bubble". Note that all of them have constant energy. Indeed:

\begin{remark}
For any $a\in\mathbb{R}^n$ and $\lambda>0$, we have
$$
\tilde{E}(U^\lambda_a)=\tilde{E}(U^1_0)=
\frac{\gamma}{n}\int_{\mathbb{R}^n}|U^1_0(x,0)|^{\frac{2n}{n-\gamma}}\,dx.
$$
\end{remark}

Now we give some notations which will be used in the following. In the half space $\mathbb{R}^{n+1}_+=\{(x,y)=(x^1,\cdots,x^n,y)\in\mathbb{R}^{n+1}\,:\,y>0\}$
we define, for $r>0$,
\begin{equation*}\begin{split}
&B^+_{r}(z_0)=\{z\in\mathbb{R}^{n+1}_+:|z-z_0|<r, z_0\in\mathbb{R}^{n+1}_+\},\\
&D_{r}(x_0)=\{x\in\mathbb{R}^n:|x-x_0|<r, x_0\in\mathbb{R}^n\},\\
&\partial'B^+_{r}(z_0)=B^+_{r}(z_0)\cap\mathbb{R}^n, \ \ \partial^+B^+_{r}(z_0)=\partial B^+_{r}(z_0)\cap\mathbb{R}^{n+1}_+.
\end{split}\end{equation*}

Fix $\gamma\in(0,1)$. Suppose that  $(X,g^+)$ is an asymptotically hyperbolic manifold with boundary $M$ satisfying, in addition, $\trace_{\hat{h}}h^{(1)}=0$ when $\gamma\in(1/2,1)$. Let $\rho$ be the special defining function given in Proposition \ref{special defining function proposition} and set $g=\rho^2g^+$, $\hat{h}=g|_{M}$. We also define
\begin{equation*}
\begin{split}
&\mathfrak{B}^+_{r}(z_0)=\{z\in X:d_g(z,z_0)<r, z_0\in \overline{X}\},\\
&\mathfrak{D}_{r}(x_0)=\{x\in M:d_{\hat{h}}(x,x_0)<r, x_0\in M\},\\
\end{split}
\end{equation*}
Now, modulo the definitions of the weighted Sobolev space $W^{1,2}(X,\rho^{1-2\gamma})$ and of a Palais-Smale sequence (see section 2), the main result of this paper is the following fractional type blow up analysis theorem:

\begin{theorem}\label{main theorem}
 Let $\{u_\alpha\geq0\}_{\alpha\in \mathbb{N}}\subset W^{1,2}(X,\rho^{1-2\gamma})$ be a Palais-Smale sequence for $\{I^{\gamma,\alpha}_g\}_{\alpha\in \mathbb{N}}$. Then there exist an integer $m\geq1$, sequences $\{\mu_\alpha^j>0\}_{\alpha\in \mathbb{N}}$ and $ \{x_\alpha^j\}_{\alpha\in \mathbb{N}}\subset M $ for $j=1, \cdots, m$, also a nonnegative solution $u^0\in W^{1,2}(X,\rho^{1-2\gamma})$ to equation \eqref{limit eq} and nontrivial nonnegative functions $U^{\lambda_j}_{a_j}\in W^{1,2}(\mathbb{R}^{n+1}_+,y^{1-2\gamma})$ for some $\lambda_j>0$ and $a_j\in \mathbb{R}^n$ as given in \eqref{Poissonformula}, satisfying, up to a subsequence,
\begin{itemize}
\item[(1)] $\mu^j_\alpha\rightarrow0$ as $\alpha\rightarrow+\infty$, for  $j=1, \cdots, m$;

\item[(2)] $ \{x_\alpha^j\}_{\alpha\in \mathbb{N}}$ converges on $M$ as $\alpha\rightarrow+\infty$, for  $j=1, \cdots, m$;

\item[(3)] As $\alpha\to+\infty$,
$$\|u_\alpha-u^0-\sum^m_{j=1}\eta^j_\alpha u^j_\alpha \|_{W^{1,2}(X,\rho^{1-2\gamma})}\rightarrow0,$$ where
$$
u^j_\alpha(z)=(\mu^j_\alpha)^{-\frac{n-2\gamma}{2}}U^{\lambda_j}_{a_j}((\mu_\alpha^j)^{-1}\varphi^{-1}_{x^j_\alpha}(z))
$$
for $z\in\varphi_{x^j_\alpha}(B^+_{r_0}(0))$, and $\varphi_{x^j_\alpha}$ are Fermi coordinates centered at $x^j_\alpha\in M$
 with $r_0>0$ small, and $\eta^j_\alpha$ are cutoff functions such that
 $$
 \eta^j_\alpha\equiv1 \ \ \hbox{in} \ \varphi_{x^j_\alpha}(B^+_{r_0}(0))
 \ \ \ \hbox{and} \ \ \
 \eta^j_\alpha\equiv0 \ \ \hbox{in} \ X\setminus\varphi_{x^j_\alpha}(B^+_{2r_0}(0)) ;
  $$

\item[(4)] The energies $$I^{\gamma,\alpha}_g(u_\alpha)-I^\infty_g(u^0)-m\tilde{E}(U^{\lambda_j}_{a_j})\rightarrow0$$ as $\alpha\rightarrow+\infty$;

\item[(5)] For any $1\leq i, j\leq m$, $i\neq j$,
$$
\frac{\mu^i_\alpha}{\mu^j_\alpha}+\frac{\mu^j_\alpha}{\mu^i_\alpha}+\frac{d_{\hat{h}}(x^i_\alpha,x^j_\alpha)^2}{\mu^i_\alpha\mu^j_\alpha}\rightarrow+\infty, \ \ \ \hbox{as}   \ \alpha\rightarrow+\infty.
$$
\end{itemize}
\end{theorem}

\begin{remark}
\begin{itemize}
\item[(i)] We call $\eta^j_\alpha u^j_\alpha$ a bubble for $j=1,\cdots,m$.

\item[(ii)] If $u_\alpha\rightarrow u^0$ strongly in $W^{1,2}(X,\rho^{1-2\gamma})$ as $\alpha\rightarrow+\infty$, then we must have $m=0$ here.
\end{itemize}
\end{remark}

Although the local case $\gamma=1$ is well known (\cite{D-H-R},\cite{S}), the most interesting point in the fractional case is the fact that one still has an energy decomposition into bubbles, and that these bubbles are non-interfering, which is surprising since our operator is non-local. And in the setting of Euclidean fractional Sobolev space, the similar global compactness results were established in \cite{P-P1},\cite{P-P2} and \cite{Y-Y-Y}.

\vskip 0.1 in
This paper is organized as follows: In section 2, we will first recall the definition of weighted Sobolev spaces and Palais-Smale sequences. Then we shall derive  a criterion for the strong convergence of a given Palais-Smale sequence. At last, $\varepsilon$-regularity estimates will be established. In section 3, we shall extract the first bubble from the Palais-Smale sequence which is not strongly convergent. In section 4, we will give the proof of Theorem \ref{main theorem}. Finally, some regularity estimates of the degenerate elliptic PDE are given as Appendix in Section 5.


\section{Preliminary Results}

Most of the arguments in this section are analogous to the results in \cite{D-H-R} (Chapter 3). For the convenience of reader, we also prove these lemmas with the necessary modifications.

From now on we use $2^*=2n/(n-2\gamma), \gamma\in(0,1)$ for simplicity and always assume that Palais-Smale sequences are all nonnegative. Moreover, the notation $o(1)$ will be taken with respect to to the limit $\alpha\to+\infty$.

\begin{definition}\label{weighted sobolev space}
The weighted Sobolev space $W^{1,2}(X,\rho^{1-2\gamma})$ is defined as the closure of $\mathcal C^{\infty}(\overline{X})$ with norm
\begin{align} \label{norm1}
\|u\|_{W^{1,2}(X,\rho^{1-2\gamma})}=\left(\int_X\rho^{1-2\gamma}|\nabla u|^2_g\,dv_g+\int_Mu^2d\sigma_{\hat{h}}\right)^{\frac{1}{2}}
\end{align}
where $dv_g$ is the volume form of the asymptotically hyperbolic Riemannian manifold $(X,g)$ and $d\sigma_{\hat{h}}$ is the volume form of the conformal infinity $(M,[\hat{h}])$.
\end{definition}

\begin{proposition}
 The norm defined above is equivalent to the following traditional norm
\begin{align} \label{norm2}
\|u\|^*_{W^{1,2}(X,\rho^{1-2\gamma})}=\left(\int_X\rho^{1-2\gamma}(|\nabla u|^2_g+u^2)\,dv_g\right)^{\frac{1}{2}}.
\end{align}
\end{proposition}

On one hand, $\|\cdot\|$ can be controlled by $\|\cdot\|^*$. This is a easy consequence of the following two propositions. The first one is a trace Sobolev embedding on Euclidean space.

\begin{proposition}\cite{J-X}
For any $u\in \mathcal C^\infty_0(\mathbb{R}^{n+1}_+)$ we have
$$
\left(\int_{\mathbb{R}^n}|u(x,0)|^{2^*}dx\right)^{\frac{2}{2^*}}
\leq S(n,\gamma)\int_{\mathbb{R}^{n+1}_+}y^{1-2\gamma}|\nabla u(x,y)|^2dxdy
$$
where
$$
S(n,\gamma)=\frac{1}{2\pi^\gamma}
\frac{\Gamma{(\gamma)}}{\Gamma(1-\gamma)}
\frac{\Gamma{(\frac{n-2\gamma}{2})}}{{\Gamma({\frac{n+2\gamma}{2}}})}
\left(\frac{\Gamma(n)}{\Gamma(n/2)}\right)^{\frac{2\gamma}{n}}.
$$
\end{proposition}
Using a standard partition of unity argument one obtains a weighted trace Sobolev inequality on an asymptotically hyperbolic manifold:

\begin{proposition}\label{weighted trace sobolev embedding}\cite{J-X}
For any $\varepsilon>0$, there exists a constant $C_\varepsilon>0$ such that
$$
\left(\int_M|u|^{2^*}\,d\sigma_{\hat{h}}\right)^{\frac{2}{2^*}}
\leq (S(n,\gamma)+\varepsilon)\int_X\rho^{1-2\gamma}|\nabla u|^2_g\,dv_g+C_\varepsilon\int_X\rho^{1-2\gamma}u^2\,dv_g.
$$
\end{proposition}

On the other hand, $\|\cdot\|^*$ can be controlled by $\|\cdot\|$, which is implied by the following proposition.
\begin{proposition}
For any $u\in W^{1,2}(X,\rho^{1-2\gamma})$, there exists a constant $C>0$ such that
$$
\int_X\rho^{1-2\gamma}u^2\,dv_g\leq C\left(\int_X\rho^{1-2\gamma}|\nabla u|^2_g\,dv_g+\int_Mu^2\,d\sigma_{\hat{h}}\right).
$$
\end{proposition}

\begin{proof}
We use a contradiction argument. Thus, assume that for any $\alpha\geq1$ there exists $u_\alpha$ satisfying
$$
\int_X\rho^{1-2\gamma}u^2_\alpha \,dv_g\geq \alpha\left(\int_X\rho^{1-2\gamma}|\nabla u_\alpha|^2_g\,dv_g+\int_Mu^2_\alpha \,d\sigma_{\hat{h}}\right).
$$
Without loss of generality, we can assume that $\int_X\rho^{1-2\gamma}u^2_\alpha \,dv_g=1$. Then we have
$$
\int_X\rho^{1-2\gamma}(|\nabla u_\alpha|^2_g+u^2_\alpha)\,dv_g\leq1+\frac{1}{\alpha}.
$$
Then there exists a weakly convergent subsequence, also denoted by $\{u_\alpha\}$, such that $u_\alpha\rightharpoonup u_0$ in $W^{1,2}(X,\rho^{1-2\gamma},\|\cdot\|^*)$.

Since
$$
\lim_{\alpha\rightarrow\infty}\int_X\rho^{1-2\gamma}|\nabla u_\alpha|^2_gdv_g=0\ \ \hbox{and}
\ \ \lim_{\alpha\rightarrow\infty}\int_Mu^2_\alpha d\sigma_{\hat{h}}=0,
$$
then we get that $u_0\equiv0$. On the other hand, via the following Proposition \ref{fractional compact embedding}, the embeddig $W^{1,2}(X,\rho^{1-2\gamma},\|\cdot\|^*)\hookrightarrow L^2(X,\rho^{1-2\gamma})$ is compact. So we have
$$
\int_X\rho^{1-2\gamma}u^2_0\,dv_g=1,
$$
which contradicts the fact that $u_0\equiv0$. Then the proof is completed.
\end{proof}

\begin{proposition}\label{fractional compact embedding}\cite{J-X,K,DiN-P-V}
Let $1\leq p\leq q<\infty$ with $\frac{1}{n+1}>\frac{1}{p}-\frac{1}{q}$.
\begin{itemize}
\item[(i)] Suppose $2-2\gamma\leq p$. Then $W^{1,p}(X,\rho^{1-2\gamma},\|\cdot\|^*)$ is compactly embedded in $L^q(X,\rho^{1-2\gamma})$ if
$$
\frac{2-2\gamma}{p(n+2-2\gamma)}>\frac{1}{p}-\frac{1}{q};
$$

\item[(ii)] Suppose $2-2\gamma> p$. Then $W^{1,p}(X,\rho^{1-2\gamma},\|\cdot\|^*)$ is compactly embedded in $L^q(X,\rho^{1-2\gamma})$ if and only if
$$
\frac{1}{(n+2-2\gamma)}>\frac{1}{p}-\frac{1}{q}.
$$
\end{itemize}
\end{proposition}

We will always use the norm in $W^{1,2}(X,\rho^{1-2\gamma})$ in the following unless otherwise stated.

\begin{definition}
$\overline{W}^{1,2}(X,\rho^{1-2\gamma})$  is the closure of $\mathcal C^\infty_0(X)$ in $W^{1,2}(X,\rho^{1-2\gamma})$ with the norm
$$
\|u\|_{\overline{W}^{1,2}(X,\rho^{1-2\gamma})}=\left(\int_X \rho^{1-2\gamma}|\nabla u|^2_g\,dv_g\right)^{\frac{1}{2}}.
$$
\end{definition}

 Now we define Palais-Smale sequences for the functional \eqref{funcional} precisely.

\begin{definition} \label{palais-smale}
$\{u_\alpha\}_{\alpha\in \mathbb{N}}\subset W^{1,2}(X,\rho^{1-2\gamma})$ is called a Palais-Smale sequence for
$\{I^{\gamma,\alpha}_g\}_{\alpha\in \mathbb{N}}$ if:
\begin{itemize}
\item[(i)] $\{I^{\gamma,\alpha}_g(u_\alpha)\}_{\alpha\in \mathbb{N}}$ is uniformly bounded; and

\item[(ii)] as $\alpha\rightarrow+\infty$,
$$DI^{\gamma,\alpha}_g(u_\alpha)\rightarrow 0\text{ strongly in }W^{1,2}(X,\rho^{1-2\gamma})',$$ where we have defined $W^{1,2}(X,\rho^{1-2\gamma})'$  as the dual space of $W^{1,2}(X,\rho^{2\gamma-1})$, i.e. for any $\phi\in W^{1,2}(X,\rho^{1-2\gamma})$, then
\begin{equation}\label{PS1}\begin{split}
DI^{\gamma,\alpha}_g(u_\alpha)\cdot\phi
=&\int_{X}\rho^{1-2\gamma}\langle\nabla u_\alpha,\nabla \phi\rangle_g\,dv_g
+\int_MQ^\gamma_\alpha u_\alpha\phi \, d\sigma_{\hat{h}}
-\int_M u_\alpha^{2^*-1}\phi \,d\sigma_{\hat{h}}\\
=&\,o(\|\phi\|_{W^{1,2}(X,\rho^{1-2\gamma})}) \ \ \hbox{as}\ \alpha\rightarrow+\infty.
\end{split}\end{equation}
\end{itemize}
\end{definition}

The main properties of Palais-Smale sequences are contained in the next several lemmas:

\begin{lemma}\label{bounded lemma}
Let $\{u_\alpha\}_{\alpha\in \mathbb{N}}\subset W^{1,2}(X,\rho^{1-2\gamma})$ be a Palais-Smale sequence for the functionals $\{I^{\gamma,\alpha}_g\}_{\alpha\in \mathbb{N}}$,
then $\{u_\alpha\}_{\alpha\in \mathbb{N}}$ is uniformly bounded in $W^{1,2}(X,\rho^{1-2\gamma})$.
\end{lemma}

\begin{proof}
We can take  $\phi=u_\alpha\in W^{1,2}(X,\rho^{1-2\gamma})$ as a test function in (ii) of Definition \ref{palais-smale}, then we get
$$
\int_X \rho^{1-2\gamma}|\nabla u_\alpha|^2_g\,dv_g+\int_MQ_\alpha^\gamma u^2_\alpha \,d\sigma_{\hat{h}}
=\int_Mu_\alpha^{2^*}d\sigma_{\hat{h}}+o(\|u_\alpha\|_{W^{1,2}(X,\rho^{1-2\gamma})}),
$$
which yields that
\begin{equation*}\begin{split}
I^{\gamma,\alpha}_g(u_\alpha)
&=\frac{1}{2}\int_X\rho^{1-2\gamma}|\nabla u_\alpha|_g^2\,dv_g
+\frac{1}{2}\int_MQ_\alpha^\gamma u_\alpha^2\,d\sigma_{\hat{h}}
-\frac{1}{2^*}\int_Mu_\alpha^{2^*}\,d\sigma_{\hat{h}}\\
&=\frac{\gamma}{n}\int_Mu_\alpha^{2^*}\,d\sigma_{\hat{h}}+o(\|u_\alpha\|_{W^{1,2}(X,\rho^{1-2\gamma})}).
\end{split}\end{equation*}
Since $\{I^{\gamma,\alpha}_g(u_\alpha)\}_{\alpha\in\mathbb{N}}$ is uniformly bounded by (i) of Definition \ref{palais-smale}, there exists a constant $C>0$ such that

$$
\int_Mu_\alpha^{2^*}d\sigma_{\hat{h}}\leq C+o(\|u_\alpha\|_{W^{1,2}(X,\rho^{1-2\gamma})}),
$$
which by H\"{o}lder's inequality yields
$$
\int_Mu_\alpha^2\,d\sigma_{\hat{h}}
 \leq C\left( \int_Mu_\alpha^{2^*}\,d\sigma_{\hat{h}}\right)^{2/2^*}
 \leq C+o(\|u_\alpha\|^{2/2^*}_{W^{1,2}(X,\rho^{1-2\gamma})}).
$$
Note that since $|Q^\gamma_\alpha|\leq C$ for some constant $C>0$, we can choose sufficiently large $C_1>0$ such that $C_1+Q^\gamma_\alpha\geq1$ on $M$. It follows
\begin{equation*}\begin{split}
\|u_\alpha\|^2_{W^{1,2}(X,\rho^{1-2\gamma})}
&= \int_X \rho^{1-2\gamma}|\nabla u_\alpha|^2_g\,dv_g+\int_Mu^2_\alpha\, d\sigma_{\hat{h}}\\
&\leq\int_X \rho^{1-2\gamma}|\nabla u_\alpha|^2_g\,dv_g+\int_MQ_\alpha^\gamma u^2_\alpha \,d\sigma_{\hat{h}}+C_1\int_ Mu^2_\alpha \,d\sigma_{\hat{h}}\\
&\leq\int_Mu^{2^*}_\alpha \,d\sigma_{\hat{h}}+o(\|u_\alpha\|_{W^{1,2}(X,\rho^{1-2\gamma})})+C
+o(\|u_\alpha\|^{2/2^*}_{W^{1,2}(X,\rho^{1-2\gamma})})\\
&\leq C+o(\|u_\alpha\|_{W^{1,2}(X,\rho^{1-2\gamma})})+o(\|u_\alpha\|^{2/2^*}_{W^{1,2}(X,\rho^{1-2\gamma})}).
\end{split}\end{equation*}
which concludes that $\{u_\alpha\}_{\alpha\in\mathbb{N}}$ is uniformly bounded in $W^{1,2}(X,\rho^{1-2\gamma})$ since $2/2^*<1$.  The proof is finished.
\end{proof}

\begin{remark}\label{remark on limit function}
From Lemma \ref{bounded lemma}, it is easy to see that there exists a function $u^0$ in $W^{1,2}(X,\rho^{1-2\gamma})$ such that
$u_\alpha \rightharpoonup u^0$ weakly in $W^{1,2}(X,\rho^{1-2\gamma})$ as $\alpha\rightarrow+\infty$.
\end{remark}

\begin{proposition}\label{u0 property}
$u^0\geq0$ in $\overline{X}$.
\end{proposition}
\begin{proof}
Using Proposition \ref{weighted trace sobolev embedding}, we can easily get that $u_\alpha\rightarrow u^0$ in $L^2(M,\hat{h})$ as $\alpha\rightarrow +\infty$, so furthermore we have $u_\alpha\rightarrow u^0$ almost everywhere on $M$. Noting that $u_\alpha\geq0$ on $M$, then we obtain that $u^0\geq0$ on $M$. On the other hand, by Proposition \ref{fractional compact embedding},
 and the equivalence of the norms $\|\cdot\|$ and $\|\cdot\|^*$, we have $u_\alpha \rightarrow u^0$ in $L^2(X,\rho^{1-2\gamma})$ as $\alpha\rightarrow+\infty$. For any $z\in X$, take $d_z<$ dist$(z,M)$, then we also have $u_\alpha\rightarrow u^0$ in $L^2(\mathfrak{B}^+_{d_z}(z),\rho^{1-2\gamma})$. Since $\rho^{1-2\gamma}$ is bounded below by a positive constant in $\mathfrak{B}^+_{d_z}(z)$,  we get $u_\alpha\rightarrow u^0$ almost everywhere in $\mathfrak{B}^+_{d_z}(z)$ up to passing to a subsequence. Noting that $u_\alpha\geq0$ in $X$, we obtain $u^0\geq0$ in $\mathfrak{B}^+_{d_z}(z)$. Since $z$ is arbitrary in $X$, then $u^0\geq0$ in $X$. Combining the above arguments, we conclude that $u\geq0$ in $\overline{X}$.
\end{proof}

Next we define the two limit functionals
$$
I^\gamma_g(u)=\frac{1}{2}\int_X\rho^{1-2\gamma}|\nabla u|_g^2\,dv_g-\frac{1}{2^*}\int_M|u|^{2^*}d\sigma_{\hat{h}}
$$
and
$$
I^{\gamma,\infty}_g(u)=\frac{1}{2}\int_X\rho^{1-2\gamma}|\nabla u|_g^2\,dv_g+\frac{1}{2}\int_MQ^\gamma_\infty u^2\,d\sigma_{\hat{h}}-\frac{1}{2^*}\int_M|u|^{2^*}d\sigma_{\hat{h}}.
$$
We have the following lemma:

\begin{lemma}\label{limit lemma}
 Let $\{u_\alpha\}_{\alpha\in \mathbb{N}}\subset W^{1,2}(X,\rho^{1-2\gamma})$ be a Palais-Smale sequence for $\{I^{\gamma,\alpha}_g\}_{\alpha\in\mathbb{N}}$, and $u_\alpha\rightharpoonup u^0$ weakly in $W^{1,2}(X,\rho^{1-2\gamma})$ as $\alpha\rightarrow+\infty$. We also denote $\hat{u}_\alpha=u_\alpha- u^0\in W^{1,2}(X,\rho^{1-2\gamma})$. Then
\begin{itemize}
\item[(i)] $u^0$ is a nonnegative weak solution to the limit equation
\begin{equation}\label{limit eq}
    \left\{
  \begin{split}
    -\divergence(\rho^{1-2\gamma}\nabla u)=0 & \ \hbox{ in } \ X,\\
    -\lim_{\rho\rightarrow0}\rho^{1-2\gamma}\partial_\rho u+Q^\gamma_\infty u=u^{2^*-1}   & \ \hbox{ on } \ M;\
  \end{split}
\right.
\end{equation}

\item[(ii)]  $I^{\gamma,\alpha}_g(u_\alpha)=I^\gamma_g(\hat{u}_\alpha)+I^{\gamma,\infty}_g(u^0)+o(1)$ as $\alpha \rightarrow+ \infty$;

\item[(iii)]  $\{\hat{u}_\alpha\}_{\alpha\in \mathbb{N}}$ is a Palais-Smale sequence for $I^\gamma_g$.
\end{itemize}
\end{lemma}

\begin{proof}
(i) As $\mathcal C^\infty(\overline{X})$ is dense in $W^{1,2}(X,\rho^{1-2\gamma})$, we only consider the proof in $\mathcal C^\infty(\overline{X})$. Let $\phi\in \mathcal C^\infty(\overline{X})$.
Since $Q^\gamma_\alpha\rightarrow Q^\gamma_\infty$ in $L^2(M,\hat{h})$ as $\alpha\rightarrow+\infty$ and $u_\alpha\rightharpoonup u^0$ weakly in $W^{1,2}(X,\rho^{1-2\gamma})$ as $\alpha\rightarrow+\infty$, then
$$
\int_MQ^\gamma_\alpha u_\alpha\phi \,d\sigma_{\hat{h}}=\int_MQ^\gamma_\infty u^0\phi \,d\sigma_{\hat{h}}+o(1).
$$
Passing to the limit in \eqref{PS1}, we get easily that
$$
\int_X\rho^{1-2\gamma}\langle\nabla u^0,\nabla\phi\rangle_g \,dv_g+\int_MQ^\gamma_\infty u^0\phi \,d\sigma_{\hat{h}}
=\int_M(u^0)^{2^*-1}\phi\, d\sigma_{\hat{h}},
$$
i.e. $u^0$ is a weak solution to the limit equation \eqref{limit eq}.\\

For the proof of (ii), recall that
$$\int_MQ^\gamma_\alpha u^2_\alpha \,d\sigma_{\hat{h}}=\int_MQ^\gamma_\infty (u^0)^2 \,d\sigma_{\hat{h}}+o(1),$$
and
\begin{equation*}
\begin{split}
I^{\gamma,\alpha}_g(u_\alpha)
&=\frac{1}{2}\int_X\rho^{1-2\gamma}|\nabla u_\alpha|_g^2\,dv_g+\frac{1}{2}\int_MQ_\alpha^\gamma u_\alpha^2\,d\sigma_{\hat{h}}
-\frac{1}{2^*}\int_Mu_\alpha^{2^*}\,d\sigma_{\hat{h}},\\
I^{\gamma,\infty}_g(u^0)
&=\frac{1}{2}\int_X\rho^{1-2\gamma}|\nabla u^0|_g^2\,dv_g+\frac{1}{2}\int_MQ_\infty^\gamma(u^0)^2\,d\sigma_{\hat{h}}
-\frac{1}{2^*}\int_M (u^0)^{2^*}\,d\sigma_{\hat{h}},\\
I^\gamma_g(\hat{u}_\alpha)
&=\frac{1}{2}\int_X\rho^{1-2\gamma}|\nabla\hat{u}_\alpha|_g^2\,dv_g
-\frac{1}{2^*}\int_M|\hat{u}_\alpha|^{2^*}\,d\sigma_{\hat{h}},
\end{split}\end{equation*}
where $\hat{u}_\alpha=u_\alpha-u^0$. Then
\begin{equation*}\begin{split}
 I^{\gamma,\alpha}_g(u_\alpha) &-I^{\gamma,\infty}_g(u^0)-I^\gamma_g(\hat{u}_\alpha)\\
=&\int_X\rho^{1-2\gamma}\langle\nabla u^0,\nabla\hat{u}_\alpha\rangle_g \,dv_g-\frac{1}{2^*}\int_M\Phi_\alpha \,d\sigma_{\hat{h}}+o(1),
\end{split}\end{equation*}
where $\Phi_\alpha=|\hat{u}_\alpha+u^0|^{2^*}-|\hat{u}_\alpha|^{2^*}-|u^0|^{2^*}$.
Note that $\hat{u}_\alpha\rightharpoonup0$ weakly in $W^{1,2}(X,\rho^{1-2\gamma})$ as $\alpha\rightarrow+\infty$, thus
$$
\int_X\rho^{1-2\gamma}\langle\nabla u^0,\nabla \hat{u}_\alpha\rangle_g\,dv_g\rightarrow0, \ \ \ \hbox{as}  \  \alpha\rightarrow\infty.
$$
On the other hand, it is easy to check that there exists a constant $C>0$, independent of $\alpha$, such that
$$
\left||\hat{u}_\alpha+u^0|^{2^*}-|\hat{u}_\alpha|^{2^*}-|u^0|^{2^*}\right|
\leq
C\left(|\hat{u}_\alpha|^{2^*-1}|u^0|+|u^0|^{2^*-1}|\hat{u}_\alpha|\right).
$$
As a consequence, since $\hat{u}_\alpha\rightharpoonup0$ weakly in $L^{2^*}(M,\hat{h})$ by Proposition \ref{weighted trace sobolev embedding}, we have
$$
\int_M|\Phi_\alpha|\,d\sigma_{\hat{h}}\rightarrow0,  \ \ \ \hbox{as}  \  \alpha\rightarrow+\infty.
$$
The proof of (ii) is completed.\\

(iii) For any $\phi\in \mathcal C^\infty(\overline{X})$, by (i) we have
$$
DI^{\gamma,\infty}_g(u^0)\cdot\phi=0.
$$
Since, in addition,
$$
\int_MQ^\gamma_\alpha u_\alpha\phi \,d\sigma_{\hat{h}}
=\int_MQ^\gamma_\infty u^0\phi \,d\sigma_{\hat{h}}+o(\|\phi\|_{W^{1,2}(X,\rho^{1-2\gamma})}),
$$
then
\begin{equation}\label{equation10}
DI^{\gamma,\alpha}_g(u_\alpha)\cdot\phi
=DI^\gamma_g(\hat{u}_\alpha)\cdot\phi-\int_M\Psi_\alpha\phi \,d\sigma_{\hat{h}}+o(\|\phi\|_{W^{1,2}(X,\rho^{1-2\gamma})}),
\end{equation}
where
$\Psi_\alpha=|\hat{u}_\alpha+u^0|^{2^*-2}(\hat{u}_\alpha+u^0)-|\hat{u}_\alpha|^{2^*-2}\hat{u}_\alpha -|u^0|^{2^*-2}u^0$, and it is easy to check that  there exits a constant $C>0$ independent of $\alpha$ such that

$$
|\Psi_\alpha|
\leq C\left(|\hat{u}_\alpha|^{2^*-2}|u^0|+|\hat{u}_\alpha\|u^0|^{2^*-2}\right).
$$
By H\"{o}lder's inequality and the fact $\hat{u}_\alpha\rightharpoonup0$ weakly in $W^{1,2}(X,\rho^{1-2\gamma})$ as $\alpha\rightarrow+\infty$, we have

\begin{equation*}\begin{split}
\int_M\Psi_\alpha\phi &\,d\sigma_{\hat{h}}\\
\leq&\left(\big\| |\hat{u}_\alpha|^{2^*-2}|u^0|\big\|_{L^{2^*/(2^*-1)}(M)}
+\big\| |\hat{u}_\alpha \|u^0|^{2^*-2}\big\|_{L^{2^*/(2^*-1)}(M)}\right)\|\phi\|_{L^{2^*}( M)}\\
=&\,o(1)\|\phi\|_{L^{2^*}(M)}.
\end{split}\end{equation*}
Thus from \eqref{equation10},
$$
DI^{\gamma,\alpha}_g(u_\alpha)\cdot\phi=DI^\gamma_g(\hat{u}_\alpha)\cdot\phi+o(1)\|\phi\|_{L^{2^*}(M)},
$$
which implies that $DI^\gamma_g(\hat{u}_\alpha)\rightarrow0$ in $W^{1,2}(X,\rho^{1-2\gamma})'$ as $\alpha\rightarrow+\infty$, since $\{u_\alpha\}_{\alpha\in\mathbb{N}}$ is a Palais-Smale sequence for $\{I^{\gamma,\alpha}_g\}_{\alpha\in\mathbb{N}}$.

Finally, from (ii), we know that $\{\hat{u}_\alpha\}_{\alpha\in\mathbb{N}}$ is a Palais-Smale sequence for $I^\gamma_g$. This completes the proof of the lemma.
\end{proof}

\vskip 0.1in

Now we give a criterion for strong convergence of Palais-Smale sequences.
First,

\begin{lemma}\label{criterion lemma}
Let $\{\hat{u}_\alpha\}_{\alpha\in \mathbb{N}}$ be a Palais-Smale sequence for $I^\gamma_g$ and such that $\hat{u}_\alpha\rightharpoonup0$ weakly in $W^{1,2}(X,\rho^{1-2\gamma})$ as $\alpha\rightarrow+\infty$. If $I^\gamma_g(\hat{u}_\alpha)\rightarrow\beta$ and
\begin{equation}\label{beta-0}\beta<\beta_0=
\frac{\gamma}{n}(d^*_\gamma)^{-\frac{n}{2\gamma}}\Lambda_\gamma(M,[\hat{h}])^{\frac{n}{2\gamma}},
\end{equation}
then $\hat{u}_\alpha\rightarrow0$ in $W^{1,2}(X,\rho^{1-2\gamma})$ as $\alpha\rightarrow+\infty$.
\end{lemma}

\begin{proof}
By Lemma \ref{bounded lemma} (here $Q^\gamma_\alpha\equiv0$), there exists a constant $C>0$ such that $\|\hat{u}_\alpha\|_{W^{1,2}(X,\rho^{1-2\gamma})}\leq C$ for all $\alpha\in\mathbb{N}$, so
\begin{equation*}\begin{split}
DI^\gamma_g(\hat{u}_\alpha)\cdot\hat{u}_\alpha
&=\int_X\rho^{1-2\gamma}|\nabla\hat{u}_\alpha|_g^2\,dv_g-\int_M|\hat{u}_\alpha|^{2^*}d\sigma_{\hat{h}}\\
&=o(\|\hat{u}_\alpha\|_{W^{1,2}(X,\rho^{1-2\gamma})})=o(1).
\end{split}\end{equation*}
Then note that $I^\gamma_g(\hat{u}_\alpha)\rightarrow\beta$ as $\alpha\rightarrow+\infty$, we have
\begin{equation}\label{equation12}\begin{split}
\beta+o(1)
&=I^\gamma_g(\hat{u}_\alpha) \\
&=\frac{1}{2}\int_X\rho^{1-2\gamma}|\nabla\hat{u}_\alpha|_g^2\,dv_g
-\frac{1}{2^*}\int_M|\hat{u}_\alpha|^{2^*}d\sigma_{\hat{h}}\\
&=\frac{\gamma}{n}\int_X\rho^{1-2\gamma}|\nabla\hat{u}_\alpha|_g^2\,dv_g+o(1)\\
&=\frac{\gamma}{n}\int_M|\hat{u}_\alpha|^{2^*}d\sigma_{\hat{h}}+o(1).
\end{split}\end{equation}

On the other hand, it was shown in \cite{G-Q} that in the positive curvature case, then the $\gamma$-Yamabe constant \eqref{Yamabe-constant} must be positive: $\Lambda_\gamma(M,[\hat{h}])>0$. Moreover, by definition,
\begin{equation}\label{equation13}
\Lambda_\gamma(M,[\hat{h}])\left(\int_M|\hat{u}_\alpha|^{2^*}\,d\sigma_{\hat{h}}\right)^{\frac{2}{2^*}}
\leq d^*_\gamma\int_X\rho^{1-2\gamma}|\nabla \hat{u}_\alpha|^2_g\,dv_g
+\int_MQ^{\hat{h}}_\gamma\hat{u}_\alpha^2\,d\sigma_{\hat{h}}.
\end{equation}
where $d^*_\gamma>0$. We also know that $|Q_\gamma^{\hat{h}}|\leq C$ on $M^n$. Note that $\hat{u}_\alpha\rightharpoonup0$ in $L^{2^*}(M,\hat{h})$ as $\alpha\rightarrow+\infty$ by Proposition \ref{weighted trace sobolev embedding}, then $\int_M\hat{u}_\alpha^2\,d\sigma_{\hat{h}}\rightarrow0$ as $\alpha\rightarrow+\infty$ since the embedding $L^{2^*}(M,\hat{h})\subset L^{2}(M,\hat{h})$ is compact. So we get from \eqref{equation12} and \eqref{equation13} that
$$
\left( \frac{n}{\gamma}\beta+o(1)\right)^{\frac{2}{2^*}}\leq d^*_\gamma\Lambda_\gamma(M,[\hat{h}])^{-1}\frac{n}{\gamma}\beta+o(1).
$$
Taking $\alpha\rightarrow+\infty$, we must have $\beta=0$ because of our initial condition \eqref{beta-0}. The Lemma is proved.
\end{proof}

\vskip 0.1in

Note that the Palais-Smale condition (ii) is the weak form of a Dirichlet-to-Neumann problem for a degenerate elliptic PDE. In fact, as $DI^\gamma_g(\hat{u}_\alpha)\rightarrow0$ in $W^{1,2}(X,\rho^{1-2\gamma})'$, it follows that, for any $\psi\in W^{1,2}(X,\rho^{1-2\gamma})$,
\begin{equation}\label{general weak solution}
\int_X\rho^{1-2\gamma}\langle\nabla\hat{u}_\alpha,\nabla\psi\rangle_g \,dv_g
-\int_M|\hat{u}_\alpha|^{2^*-2}\hat{u}_\alpha\psi \,d\sigma_{\hat{h}}
=o(1)\|\psi\|_{W^{1,2}(X,\rho^{1-2\gamma})}.
\end{equation}
In particular, for any $\bar{\psi}\in \overline{W}^{1,2}(X,\rho^{1-2\gamma})$, then
$$
\int_X\rho^{1-2\gamma}\langle\nabla\hat{u}_\alpha,\nabla\bar{\psi}\rangle_g \, dv_g=o(1)\|\bar{\psi}\|_{\overline{W}^{1,2}(X,\rho^{1-2\gamma})},
$$
which is is precisely the weak formulation for the asymptotic equation
\begin{equation}\label{interior equation}
-\divergence(\rho^{1-2\gamma}\nabla\hat{u}_\alpha)=o(1) \ \  \hbox{in} \  X.
\end{equation}
Multiplying both sides of (\ref{interior equation}) by $\psi\in W^{1,2}(X,\rho^{1-2\gamma})$ and integrating by parts, we obtain that
$$
\int_M\lim_{\rho\rightarrow0}\rho^{1-2\gamma}\partial_\rho\hat{u}_\alpha\psi \,d\sigma_{\hat{h}}
+\int_X\rho^{1-2\gamma}\langle\nabla\hat{u}_\alpha,\nabla\psi\rangle_g \,dv_g
=o(1)\|\psi\|_{W^{1,2}(X,\rho^{1-2\gamma})},
$$
which combined with \eqref{general weak solution} yields that
$$
\int_M\lim_{\rho\rightarrow0}\rho^{1-2\gamma}\partial_\rho\hat{u}_\alpha\psi \,d\sigma_{\hat{h}}
+\int_M|\hat{u}_\alpha|^{2^*-2}\hat{u}_\alpha\psi \,d\sigma_{\hat{h}}
=o(1)\|\psi\|_{W^{1,2}(X,\rho^{1-2\gamma})},
$$
and this is precisely the boundary equation in the weak sense
\begin{equation}\label{boundary equation}
-\lim_{\rho\rightarrow0}\rho^{1-2\gamma}\partial_\rho\hat{u}_\alpha=
|\hat{u}_\alpha|^{2^*-2}\hat{u}_\alpha+o(1)\  \hbox{on} \  M.
\end{equation}
For the above equations \eqref{interior equation} and \eqref{boundary equation} for $\{\hat{u}_\alpha\}_{\alpha\in\mathbb{N}}$, we have the following energy estimate, which will plays an important role in the proof of the strong convergence in the next section. We use the notation $\mathfrak{B}^+_r$ instead of $\mathfrak{B}^+_r(0)$ for convenience.

\begin{lemma}\label{epsilon regularity lemma}($\varepsilon$-regularity estimates) Suppose that $\{v_\alpha\}_{\alpha\in\mathbb{N}}$ satisfies the following asymptotic boundary value problem
\begin{equation}\label{asymptotic equation}
\left\{
\begin{split}
-\divergence(\rho^{1-2\gamma}\nabla v_\alpha)=o(1) & \quad  \hbox{in}  \ \  X,\\
-\lim_{\rho\rightarrow0}\rho^{1-2\gamma}\partial_\rho v_\alpha=|v_\alpha|^{2^*-2}v_\alpha+o(1)&\quad   \hbox{on} \ \  M.
\end{split} \right.
\end{equation}
If there exists small $\varepsilon>0$ depending on $n,\gamma$ such that $\int_{\partial'\mathfrak{B}^+_{2r}}|v_\alpha|^{2^*}d\sigma_{\hat{h}}\leq\varepsilon$ uniformly in $\alpha$ for some small $r>0$, then
$$
\int_{\mathfrak{B}^+_r}\rho^{1-2\gamma}|\nabla v_\alpha|^2_g\,dv_g
\leq\frac{C}{r^2}\int_{\mathfrak{B}^+_{2r}}\rho^{1-2\gamma}v^2_\alpha \,dv_g
+C\int_{\partial'\mathfrak{B}^+_{2r}}v^2_\alpha \,d\sigma_{\hat{h}}
+o(1)\int_{\mathfrak{B}^+_{2r}}|v_\alpha|\, dv_g,
$$
where $C=C(n,\varepsilon,\gamma)$ independent of $\alpha$.
\end{lemma}

\begin{proof} Let $\eta$ be a smooth cutoff function in $\overline{X}$ such that $0\leq\eta\leq1$, $\eta\equiv1$ in $\mathfrak{B}^+_r$ and $\eta\equiv0$ in $\overline{X}\setminus \mathfrak{B}^+_{2r}$. And we also have $|\nabla \eta|\leq C/r$ in $\mathfrak{B}^+_{2r}\setminus \mathfrak{B}^+_r$. Multiplying both sides of the first equation in \eqref{asymptotic equation}
by $\eta^2v_\alpha$, integrating by parts and substituting the second equation in \eqref{asymptotic equation}, we get
\begin{equation*}\begin{split}
\int_{\mathfrak{B}^+_{2r}}\rho^{1-2\gamma}&\langle\nabla v_\alpha,\nabla(\eta^2v_\alpha)\rangle_g\,dv_g\\
&=-\int_{\partial'\mathfrak{B}^+_{2r}}\lim_{\rho\rightarrow0}\rho^{1-2\gamma}(\partial_\rho v_\alpha)\eta^2v_\alpha \,d\sigma_{\hat{h}}+o(1)\int_{\mathfrak{B}^+_{2r}}\eta^2v_\alpha\, dv_g\\
&=\int_{\partial'\mathfrak{B}^+_{2r}}\eta^2|v_\alpha|^{2^*} d\sigma_{\hat{h}}+o(1)\int_{\mathfrak{B}^+_{2r}}\eta^2 v_\alpha\, dv_g,
\end{split}\end{equation*}
so we have
\begin{equation*}\begin{split}
\int_{\mathfrak{B}^+_{2r}}\rho^{1-2\gamma}\eta^2|\nabla v_\alpha|^2_g\,dv_g
=&-\int_{\mathfrak{B}^+_{2r}}\rho^{1-2\gamma}2\eta v_\alpha\langle\nabla v_\alpha,\nabla\eta\rangle_g \,dv_g\\
&+\int_{\partial'\mathfrak{B}^+_{2r}}\eta^2|v_\alpha|^{2^*} d\sigma_{\hat{h}}
+o(1)\int_{\mathfrak{B}^+_{2r}}\eta^2v_\alpha\, dv_g\\
\leq &\,\frac{1}{2}\int_{\mathfrak{B}^+_{2r}}\eta^2\rho^{1-2\gamma}|\nabla v_\alpha|^2_g\,dv_g
+2\int_{\mathfrak{B}^+_{2r}}\rho^{1-2\gamma}|\nabla \eta|^2_g\,v^2_\alpha \,dv_g\\
&+\int_{\partial'\mathfrak{B}^+_{2r}}\eta^2|v_\alpha|^{2^*} \,d\sigma_{\hat{h}}+o(1)\int_{\mathfrak{B}^+_{2r}}\eta^2|v_\alpha|\, dv_g,
\end{split}\end{equation*}
which implies that
\begin{equation*}\begin{split}
\int_{\mathfrak{B}^+_{2r}}\rho^{1-2\gamma}\eta^2|\nabla v_\alpha|^2_g\,dv_g
\leq&4\int_{\mathfrak{B}^+_{2r}}\rho^{1-2\gamma}|\nabla \eta|^2_gv^2_\alpha\, dv_g
+2\int_{\partial'\mathfrak{B}^+_{2r}}\eta^2|v_\alpha|^{2^*}\, d\sigma_{\hat{h}}\\
&+o(1)\int_{\mathfrak{B}^+_{2r}}\eta^2|v_\alpha|\, dv_g\\
\leq&\frac{C}{r^2}\int_{\mathfrak{B}^+_{2r}}\rho^{1-2\gamma}v^2_\alpha\, dv_g
+2\int_{\partial'\mathfrak{B}^+_{2r}}(\eta v_\alpha)^2|v_\alpha|^{2^*-2}\,d\sigma_{\hat{h}}\\
&+o(1)\int_{\mathfrak{B}^+_{2r}}\eta^2|v_\alpha|\, dv_g.
\end{split}\end{equation*}
By H\"{o}lder's inequality and our initial hypothesis we have
\begin{equation*}\begin{split}
\int_{\partial'\mathfrak{B}^+_{2r}}(\eta v_\alpha)^2|v_\alpha|^{2^*-2}\, d\sigma_{\hat{h}}
&\leq \left(\int_{\partial'\mathfrak{B}^+_{2r}}|\eta v_\alpha|^{2^*}\, d\sigma_{\hat{h}}\right)^{\frac{2}{2^*}}
\left(\int_{\partial'\mathfrak{B}^+_{2r}}|v_\alpha|^{2^*} \,d\sigma_{\hat{h}}\right)^{\frac{2^*-2}{2^*}}\\
&\leq \varepsilon^{\frac{2^*-2}{2^*}}\left(\int_{\partial'\mathfrak{B}^+_{2r}}|\eta v_\alpha|^{2^*} \,d\sigma_{\hat{h}}\right)^{\frac{2}{2^*}}.
\end{split}\end{equation*}
Then it follows from above that
\begin{equation*}\begin{split}
\int_{\mathfrak{B}^+_{2r}}\rho^{1-2\gamma}|\nabla (\eta v_\alpha)|^2_g\,dv_g
\leq & 2\int_{\mathfrak{B}^+_{2r}}\rho^{1-2\gamma}(|\nabla \eta |^2_g\,v^2_\alpha+\eta^2|\nabla  v_\alpha|^2_g)\,dv_g\\
\leq & \frac{C}{r^2}\int_{\mathfrak{B}^+_{2r}}\rho^{1-2\gamma}v^2_\alpha \,dv_g
+C\varepsilon^{\frac{2^*-2}{2^*}}\left(\int_{\partial'\mathfrak{B}^+_{2r}}|\eta v_\alpha|^{2^*} \,d\sigma_{\hat{h}}\right)^{\frac{2}{2^*}}\\
&+o(1)\int_{\mathfrak{B}^+_{2r}}\eta^2|v_\alpha|\, dv_g.
\end{split}\end{equation*}
The trace Sobolev inequality on our manifold setting (Proposition \ref{weighted trace sobolev embedding}) gives that
$$
\left(\int_{\partial'\mathfrak{B}^+_{2r}}|\eta v_\alpha|^{2^*} \,d\sigma_{\hat{h}}\right)^{\frac{2}{2^*}}
\leq C\int_{\mathfrak{B}^+_{2r}}\rho^{1-2\gamma}|\nabla (\eta v_\alpha)|^2_g\,dv_g
+C\int_{\partial'\mathfrak{B}^+_{2r}}(\eta v_\alpha)^2\,d\sigma_{\hat{h}}.
$$
Therefore we obtain
\begin{equation*}\begin{split}
\int_{\mathfrak{B}^+_{2r}}\rho^{1-2\gamma}|\nabla (\eta v_\alpha)|^2_g\,dv_g
\leq &\frac{C}{r^2}\int_{\mathfrak{B}^+_{2r}}\rho^{1-2\gamma}v^2_\alpha \,dv_g
+C\varepsilon^{\frac{2^*-2}{2^*}}\int_{\mathfrak{B}^+_{2r}}\rho^{1-2\gamma}|\nabla (\eta v_\alpha)|^2_g\,dv_g\\
&+C\varepsilon^{\frac{2^*-2}{2^*}}\int_{\partial'\mathfrak{B}^+_{2r}}(\eta v_\alpha)^2\,d\sigma_{\hat{h}}
+o(1)\int_{\mathfrak{B}^+_{2r}}\eta^2|v_\alpha| \,dv_g.
\end{split}\end{equation*}
 Now we fix $r>0$ small such that $\varepsilon$ small enough satisfying $C\varepsilon^{\frac{2^*-2}{2^*}}\leq1/2$. Then we get
$$
\int_{\mathfrak{B}^+_r}\rho^{1-2\gamma}|\nabla v_\alpha|^2_g\,dv_g
\leq\frac{C}{r^2}\int_{\mathfrak{B}^+_{2r}}\rho^{1-2\gamma}v^2_\alpha \,dv_g
+C\int_{\partial'\mathfrak{B}^+_{2r}}v^2_\alpha\, d\sigma_{\hat{h}}
+o(1)\int_{\mathfrak{B}^+_{2r}}|v_\alpha| \,dv_g.
$$
This completes the proof of the lemma.
\end{proof}


\section{The First Bubble Argument}

In this section, we focus on the blow up analysis of a Palais-Smale sequence which is not strongly convergent. In particular, using the $\varepsilon$-regularity estimates (Lemma \ref{epsilon regularity lemma}), we can figure out the first bubble. We will also show that the Palais-Smale sequence obtained by subtracting a bubble is also Palais-Smale sequence and that the energy is splitting.

\begin{lemma}\label{first bubble lemma}
Let $\{\hat{u}_\alpha\}_{\alpha\in \mathbb{N}}$ be a Palais-Smale sequence for $I^\gamma_g$ such that $\hat{u}_\alpha\rightharpoonup0$ weakly in
$W^{1,2}(X,\rho^{1-2\gamma})$, but not strongly as $\alpha\rightarrow+\infty$. Then there exist a sequence of real numbers $\{\mu_\alpha>0\}_{\alpha\in \mathbb{N}}$,
$\mu_\alpha\rightarrow 0$ as $\alpha\rightarrow+\infty$, a converging sequence of points $\{x_\alpha\}_{\alpha\in \mathbb{N}}\subset M$ and a nontrivial solution $u$ to the equation

\begin{equation}\label{Liouville equation}
    \left\{
  \begin{split}
    -\divergence(y^{1-2\gamma}\nabla u)=0 & \quad\hbox{ in } \ \mathbb{R}^{n+1}_+,\\
    -\lim_{y\rightarrow0}y^{1-2\gamma}\partial_y u=|u|^{2^*-2}u & \quad \hbox{ on } \ \mathbb{R}^n,\
  \end{split}
\right.
\end{equation}
such that, up to a subsequence, if we take
$$
\hat{v}_\alpha(z)
=\hat{u}_\alpha(z)-{\eta}_\alpha(z)\mu_\alpha^{-\frac{n-2\gamma}{2}}
u(\mu_\alpha^{-1}\varphi_{x_\alpha}^{-1}(z)), \ \ z\in\varphi_{x_\alpha}(B^+_{2r_0}(0))
$$
where $r_0$, $\eta_\alpha$ and $\varphi_{x_\alpha}$ are as same as in the Theorem \ref{main theorem}, then we have the following three conclusions
\begin{itemize}
\item[(i)] $\hat{v}_\alpha\rightharpoonup0$ weakly in $W^{1,2}(X,\rho^{1-2\gamma})$ as $\alpha\rightarrow+\infty$;

\item[(ii)] $\{\hat{v}_\alpha\}_{\alpha\in \mathbb{N}}$ is also a Palais-Smale sequence for $I^\gamma_g$;

\item[(iii)] $I^\gamma_g(\hat{v}_\alpha)=I^\gamma_g(\hat{u}_\alpha)-\tilde{E}(u)+o(1)$ as $\alpha\rightarrow+\infty$.
\end{itemize}
\end{lemma}

\begin{proof}
Without loss of generality, we assume that $\hat{u}_\alpha\in \mathcal C^\infty(\overline{X})$. By the proof of Lemma \ref{criterion lemma},
$$
I^\gamma_g(\hat{u}_\alpha)=\frac{\gamma}{n}\int_X\rho^{1-2\gamma}|\nabla\hat{u}_\alpha|_g^2\,dv_g+o(1)
=\frac{\gamma}{n}\int_M|\hat{u}_\alpha|^{2^*}d\sigma_{\hat{h}}+o(1).
$$
Note that $\{\hat{u}_\alpha\}_{\alpha\in\mathbb{N}}$ is uniformly bounded in $W^{1,2}(X,\rho^{1-2\gamma})$ by Lemma \ref{bounded lemma}, so there exist a subsequence, also denoted by  $\{\hat{u}_\alpha\}_{\alpha\in\mathbb{N}}$ and a nonnegative constant $\beta$, such that
$$
I^\gamma_g(\hat{u}_\alpha)=\beta+o(1),  \ \ \ \hbox{as} \  \alpha\rightarrow +\infty.
$$
Since $\hat{u}_\alpha\rightharpoonup0$ weakly in $W^{1,2}(X,\rho^{1-2\gamma})$ but not strongly as $\alpha\rightarrow+\infty$, by Lemma \ref{criterion lemma} again we get
$$
\lim_{\alpha\rightarrow+\infty}\int_M|\hat{u}_\alpha|^{2^*}d\sigma_{\hat{h}}
=\frac{n}{\gamma}\beta\geq\frac{n}{\gamma}\beta_0.
$$
We will decompose the rest of the proof into several steps:\\

 {\bf Step 1}. Pick up the likely blow up points. First we show the following claim.

\begin{claim}
For any $t_0>0$ small, there exist $x_0\in M$ and $\varepsilon_0>0$ such that, up to a subsequence
$$
\int_{\mathfrak{D}_{t_0}(x_0)}|\hat{u}_\alpha|^{2^*}\,d\sigma_{\hat{h}}\geq\varepsilon_0.
$$
\end{claim}

\begin{proof}
If the Claim is not true, there exists $t>0$ small, such that for any $x\in M$ it holds
$$
\int_{\mathfrak{D}_{t}(x)}|\hat{u}_\alpha|^{2^*}\,d\sigma_{\hat{h}}\rightarrow0, \ \ \alpha\rightarrow+\infty.
$$
On the other hand, since $(M,\hat h)$ is compact and $M\subset\cup_{x\in M}\mathfrak{D}_{t}(x)$, there exists an integer $N(\geq1)$ such that  $M\subset\cup^N_{i=1}\mathfrak{D}_{t}(x_i)$. Thus
$$
\int_M|\hat{u}_\alpha|^{2^*}\,d\sigma_{\hat{h}}
\leq\sum^N_{i=1}\int_{\mathfrak{D}_{t}(x_i)}|\hat{u}_\alpha|^{2^*}\,d\sigma_{\hat{h}}\rightarrow0,
\ \ \alpha\rightarrow+\infty,
$$
which is a contradiction.
\end{proof}

For $t>0$, we set
$$
\omega_\alpha(t)=\max_{x\in M}\int_{\mathfrak{D}_{t}(x)}|\hat{u}_\alpha|^{2^*}\,d\sigma_{\hat{h}}.
$$
Then by Claim 1, there exists $x_\alpha\in M$ such that
$$
\omega_\alpha(t_0)=\int_{\mathfrak{D}_{t_0}(x_\alpha)}|\hat{u}_\alpha|^{2^*}d\sigma_{\hat{h}}
\geq\varepsilon_0.
$$
Note that
$$
\int_{\mathfrak{D}_{t}(x_\alpha)}|\hat{u}_\alpha|^{2^*}d\sigma_{\hat{h}}\to0, \ \ \hbox{as} \ t\to0.
$$
Hence for any $\varepsilon\in(0,\varepsilon_0)$, there exists $t_\alpha\in(0,t_0)$ such that
\begin{equation}\label{equation20}
\varepsilon
=\int_{\mathfrak{D}_{t_\alpha}(x_\alpha)}|\hat{u}_\alpha|^{2^*}\,d\sigma_{\hat{h}}.
\end{equation}

 {\bf Step 2.} At each likely blow up point, we will establish weak convergence of a Palais-Smale sequence after properly rescaling.

For $r_0>0$ small, consider the Fermi coordinates at the likely blow up point $x_\alpha\in M$, $ \varphi_{x_\alpha}:B^+_{2r_0}(0)\rightarrow X$. Here we restrict $r_0$ to $r_0\leq i_g(X)/2$, where $i_g(X)$ is the injectivity radius of $X$.
Then for any $0<\mu_\alpha\leq1$, we define
$$
\tilde{u}_\alpha(z)=\mu_\alpha^{\frac{n-2\gamma}{2}}\hat{u}_\alpha(\varphi_{x_\alpha}(\mu_\alpha z)), \ \
\tilde{g}_\alpha(z)=(\varphi^*_{x_\alpha}g)(\mu_\alpha z),\ \ \tilde{h}_\alpha(x)=(\varphi^*_{x_\alpha}\hat{h})(\mu_\alpha x),
$$
if  $z\in B^+_{\mu^{-1}_\alpha r_0}(0)$ and $x\in D_{\mu^{-1}_\alpha r_0}(0)$.

Given $z_0\in\mathbb{R}^{n+1}_+$ and $r>0$ such that $|z_0|+r<\mu_\alpha^{-1}r_0$, we have
$$
\int_{B^+_r(z_0)}\tilde{\rho}^{1-2\gamma}_\alpha|\nabla\tilde{u}_\alpha|^2_{\tilde{g}_\alpha}
\,dv_{\tilde{g}_\alpha}
=\int_{\varphi_{x_\alpha}(\mu_\alpha B^+_r(z_0))}\rho^{1-2\gamma}|\nabla \hat{u}_\alpha|^2_g\,dv_g
$$
where
$$\tilde{\rho}_\alpha(z)=\mu^{-1}_\alpha\rho(\varphi_{x_\alpha}(\mu_\alpha z))$$
and $|d\tilde{\rho}_\alpha|_{\tilde{g}_\alpha}=1$ on $\partial'B^+_{r}(z_0)$
since $|d\rho|_g=1$ on $M$.

On the other hand, if $z_0\in\mathbb{R}^n$, and $|z_0|+r< \mu^{-1}_\alpha r_0$, then
\begin{equation*}\begin{split}
\int_{D_r(z_0)}|\tilde{u}_\alpha|^{2^*}\,d\sigma_{\tilde{h}_\alpha}
&=\int_{\varphi_{x_\alpha}(\mu_\alpha D_r(z_0))}|\hat{u}_\alpha|^{2^*}\,d\sigma_{\hat{h}}\\
&\leq \int_{\mathfrak{D}_{2\mu_\alpha r}(\varphi_{x_\alpha}(\mu_\alpha z_0))}|\hat{u}_\alpha|^{2^*}\,d\sigma_{\hat{h}}.
\end{split}\end{equation*}
Here we have used that $\varphi_{x_\alpha}(\mu_\alpha D_r(z_0))=\varphi_{x_\alpha}(D_{\mu_\alpha r}(\mu_\alpha z_0))$, and that for $|x|<r_0, |y|<r_0$, $x,y\in\mathbb{R}^n$, we have $1/2|x-y|\leq d_g(\varphi_{x_\alpha}(x),\varphi_{x_\alpha}(y))\leq2|x-y|$.

Next, take $r\in(0,r_0)$ and choose $t_0$ in Claim 1 such that $0<t_0\leq 2r$. For any $\varepsilon\in(0,\varepsilon_0)$, $\varepsilon$ to be determined later, and $t_\alpha\in(0,t_0)$, let
$0<\mu_\alpha=\frac{1}{2}r^{-1}t_\alpha\leq\frac{1}{2}r^{-1}t_0\leq1$, then by the definition of $\varepsilon$ from \eqref{equation20}, if
$|z_0|+r<\mu^{-1}_\alpha r_0$, we have
\begin{align}\label{energy bound}
\int_{\partial'B^+_r(z_0)}|\tilde{u}_\alpha|^{2^*}d\sigma_{\tilde{h}_\alpha}\leq\varepsilon.
\end{align}
Note that $ \varphi_{x_\alpha}(D_{2r\mu_\alpha}(0))=\mathfrak{D}_{t_\alpha}(x_\alpha)$, we have
\begin{equation*}\begin{split}
\varepsilon
&=\int_{\mathfrak{D}_{t_\alpha}(x_\alpha)}|\hat{u}_\alpha|^{2^*}\,d\sigma_{\hat{h}}
=\int_{\varphi_{x_\alpha}(D_{2r\mu_\alpha }(0))}|\hat{u}_\alpha|^{2^*}\,d\sigma_{\hat{h}}\\
&=\int_{\varphi_{x_\alpha}(\mu_\alpha D_{2r}(0))}|\hat{u}_\alpha|^{2^*}\,d\sigma_{\hat{h}}
=\int_{D_{2r}(0)}|\tilde{u}_\alpha|^{2^*}\,d\sigma_{\tilde{h}_\alpha}.
\end{split}\end{equation*}

Here $r_0>0$ can be chosen smaller again, such that for any $0<\mu\leq1$ and any $x_0\in M$, we can assume that
\begin{equation}\label{volume form}
\begin{split}
\frac{1}{2}\int_{\mathbb{R}^{n+1}_+}y^{1-2\gamma}|\nabla u|^2\,dxdy
\leq&\int_{\mathbb{R}^{n+1}_+}\tilde{\rho}^{1-2\gamma}_{x_0,\mu}|\nabla u|^2_{\tilde{g}_{x_0,\mu}}\,dv_{\tilde{g}_{x_0,\mu}}\\
\leq&2\int_{\mathbb{R}^{n+1}_+}y^{1-2\gamma}|\nabla u|^2\,dxdy,
\end{split}
\end{equation}
where  $u\in \overline{W}^{1,2}(\mathbb{R}^{n+1}_+,y^{1-2\gamma})$, supp$(u)\subset B^+_{2\mu^{-1}r_0}(0)$, $\tilde{\rho}_{x_0,\mu}(z)=\mu^{-1}\rho(\varphi_{x_0}(\mu z))$ and $\tilde{g}_{x_0,\mu}(z)=(\varphi^*_{x_0}g)(\mu z)$. And for $u\in L^1( \mathbb{R}^n)$ such that supp$(u)\subset D_{2\mu^{-1}r_0}(0)$, we can also assume that
$$
\frac{1}{2}\int_{\mathbb{R}^n}|u|\,dx
\leq\int_{\mathbb{R}^n}|u|\,d\sigma_{\tilde{h}_{x_0,\mu}}
\leq2\int_{\mathbb{R}^n}|u|\,dx,
$$
where $\tilde{h}_{x_0,\mu}(x)=(\varphi^*_{x_0}\hat{h})(\mu x)$.

Let  $\tilde{\eta}\in \mathcal C^\infty_0(\mathbb{R}^{n+1}_+)$ be a cutoff function satisfying
\begin{equation}\label{cutoff function}
\left\{
\begin{split}
&0\leq\tilde{\eta}\leq1,\\
&\tilde{\eta}\equiv1 \ \  \hbox{in} \ B_{1/4}^+(0),\\
&\tilde{\eta}\equiv0 \ \  \hbox{in} \ \mathbb{R}^{n+1}_+\setminus B_{3/4}^+(0).\\
\end{split}
\right.
\end{equation}
Then we set $\tilde{\eta}_\alpha(z)=\tilde{\eta}(r^{-1}_0\mu_\alpha z)$.

\begin{claim}
$\{\tilde{\eta}_\alpha\tilde{u}_\alpha\}_{\alpha\in\mathbb{N}}$ is uniformly bounded in $W^{1,2}(\mathbb{R}^{n+1}_+,y^{1-2\gamma})$.
\end{claim}

\begin{proof}
Note that
\begin{equation*}\begin{split}
\int_{\mathbb{R}^{n+1}_+}&\tilde{\rho}_\alpha^{1-2\gamma}
|\nabla(\tilde{\eta}_\alpha\tilde{u}_\alpha)|^2_{\tilde{g}_\alpha}\,dv_{\tilde{g}_\alpha}
+\int_{\mathbb{R}^{n+1}_+}\tilde{\rho}_\alpha^{1-2\gamma}(\tilde{\eta}_\alpha\tilde{u}_\alpha)^2
\,dv_{\tilde{g}_\alpha}\\
\leq&\int_{\mathbb{R}^{n+1}_+}\tilde{\rho}_\alpha^{1-2\gamma}
(2|\nabla\tilde{\eta}_\alpha|^2_{\tilde{g}_\alpha}+\tilde{\eta}^2_\alpha)
\tilde{u}^2_\alpha\, dv_{\tilde{g}_\alpha}
+2\int_{\mathbb{R}^{n+1}_+}\tilde{\rho}_\alpha^{1-2\gamma}\tilde{\eta}^2_\alpha
|\nabla\tilde{u}_\alpha|^2_{\tilde{g}_\alpha}\,dv_{\tilde{g}_\alpha}\\
\leq& C\int_X\rho^{1-2\gamma}\hat{u}_\alpha^2\,dv_g
+C\int_X\rho^{1-2\gamma}|\nabla\hat{u}_\alpha|^2_g\,dv_g\leq C,
\end{split}\end{equation*}
since $\{\hat{u}_\alpha\}_{\alpha\in\mathbb{N}}$ is uniformly bounded in $W^{1,2}(X,\rho^{1-2\gamma})$. Combining this with \eqref{volume form}, we obtain that $\{\tilde{\eta}_\alpha\tilde{u}_\alpha\}_{\alpha\in\mathbb{N}}$ is uniformly bounded in $W^{1,2}(\mathbb{R}^{n+1}_+,y^{1-2\gamma})$, as desired.
\end{proof}

Due to the weak compactness of $W^{1,2}(\mathbb{R}^{n+1}_+,y^{1-2\gamma})$, there exists some
$u$ in $W^{1,2}(\mathbb{R}^{n+1}_+,y^{1-2\gamma})$ such that
$\tilde{\eta}_\alpha\tilde{u}_\alpha\rightharpoonup u$ in
$W^{1,2}(\mathbb{R}^{n+1}_+,y^{1-2\gamma})$ as $\alpha\rightarrow+\infty$.

\vskip 0.1in

{\bf Step 3.} The weak convergence is in fact strong via $\varepsilon$-regularity estimates.

\begin{claim}
There exists $\varepsilon_1=\varepsilon_1(\gamma,n)\in(0,\varepsilon_0)$ such that for any $0<r<r_0/8$, we have $\tilde{\eta}_\alpha\tilde{u}_\alpha\rightarrow u$
in $W^{1,2}(B^+_{2r}(0),y^{1-2\gamma})$ as $\alpha\rightarrow+\infty$.
\end{claim}

\begin{proof}
Given $r$ sufficiently small, to be determined later, for any $z_0\in\mathbb{R}^{n+1}_+$,
let $\psi\in \mathcal C^\infty_0(B^+_r(z_0))\cap W^{1,2}(\mathbb{R}^{n+1}_+,y^{1-2\gamma})$. Let $\hat{\psi}_\alpha(z)=\mu_\alpha^{-\frac{n-2\gamma}{2}}\psi(\mu^{-1}_\alpha\varphi^{-1}_{x_\alpha}(z))$ for $z\in\varphi_{x_\alpha}(B^+_r(z_0))$.
Since $\{\hat{u}_\alpha\}$ satisfies the asymptotic equation \eqref{interior equation}, then we have
\begin{equation*}\begin{split}
o(1)\|\psi\|_{\overline{W}^{1,2}(\mathbb{R}^{n+1}_+,y^{1-2\gamma})}
&=o(1)\|\hat{\psi}_\alpha\|_{\overline{W}^{1,2}(X,\rho^{1-2\gamma})}\\
&=\int_{\varphi_{x_\alpha}(\mu_\alpha B^+_r(z_0))}\rho^{1-2\gamma}\langle\nabla\hat{u}_\alpha,\nabla\hat{\psi}_\alpha\rangle_g \,dv_g\\
&=\int_{B^+_r(z_0)}(\mu^{-1}_\alpha\rho)^{1-2\gamma}
\langle\nabla(\tilde{\eta}_\alpha\tilde{u}_\alpha),\nabla\psi\rangle_{\tilde{g}_\alpha}
\,dv_{\tilde{g}_\alpha},
\end{split}\end{equation*}
Here we need $|z_0|+r<1/4\mu^{-1}_\alpha r_0$ since $\tilde{\eta}_\alpha\equiv1$ in $B^+_{1/4\mu^{-1}_\alpha r_0}(0)$ by (\ref{cutoff function}).

It is easy to check that $\mu^{-1}_\alpha\rho \rightarrow y$ as $\alpha\rightarrow+\infty$ since  $|d(\mu^{-1}_\alpha\rho)|_{\tilde{g}_\alpha}=1$ on $\mathbb{R}^n$ and ${\tilde{g}_\alpha}\rightarrow (|dx|^2+dy^2)$. Then we have the asymptotic equation
\begin{equation}\label{rescaling interior equation}
 -\divergence(y^{1-2\gamma}\nabla (\tilde{\eta}_\alpha\tilde{u}_\alpha))=o(1) \ \ \hbox{in} \ \ B^+_r(z_0).
\end{equation}
Since $\tilde{\eta}_\alpha \tilde{u}_\alpha\rightharpoonup u$ weakly in $W^{1,2}(\mathbb{R}^{n+1}_+,y^{1-2\gamma})$, we simultaneously get that
\begin{equation}\label{Euclidean interior equation}
-\divergence(y^{1-2\gamma}\nabla u)=0 \ \ \hbox{in} \ \ B^+_r(z_0).
\end{equation}

Now let $\psi\in W^{1,2}(B^+_r(z_0),y^{1-2\gamma})$. Then multiplying both sides of equation \eqref{rescaling interior equation} by $\psi$ and integrating by parts, we get
\begin{equation}\label{equation21}\begin{split}
o(1)\|\psi\|_{W^{1,2}(B^+_r(z_0),y^{1-2\gamma})}
=&\int_{\partial'B^+_r(z_0)}
\lim_{y\rightarrow0}y^{1-2\gamma}\partial_y(\tilde{\eta}_\alpha\tilde{u}_\alpha)\psi \,d\sigma_{\tilde{h}_\alpha}\\
&+\int_{B^+_r(z_0)}y^{1-2\gamma}
\langle\nabla(\tilde{\eta}_\alpha\tilde{u}_\alpha),\nabla\psi\rangle_{\tilde{g}_\alpha}
\,dv_{\tilde{g}_\alpha}.
\end{split}\end{equation}
On the other hand, using \eqref{interior equation} and \eqref{boundary equation}, and the definition of $\hat{\psi}_\alpha$, we have
\begin{equation}\label{equation22}\begin{split}
\int_{B^+_r(z_0)}&y^{1-2\gamma}
\langle\nabla(\tilde{\eta}_\alpha\tilde{u}_\alpha),\nabla\psi\rangle_{\tilde{g}_\alpha}
\,dv_{\tilde{g}_\alpha}\\
=&\int_{\varphi_{x_\alpha}(\mu_\alpha B^+_r(z_0))}\rho^{1-2\gamma}\langle\nabla\hat{u}_\alpha,\nabla\hat{\psi}_\alpha\rangle_g\,dv_g\\
=&-\int_M\lim_{\rho\rightarrow0}\rho^{1-2\gamma}(\partial_\rho\hat{u}_\alpha)\hat{\psi}_\alpha \,d\sigma_{\hat{h}}
+o(1)\|\hat{\psi}_\alpha\|_{W^{1,2}(X,\rho^{1-2\gamma})}\\
=&\int_M|\hat{u}_\alpha|^{2^*-2}\hat{u}_\alpha\hat{\psi}_\alpha \,d\sigma_{\hat{h}}
+o(1)\|\hat{\psi}_\alpha\|_{W^{1,2}(X,\rho^{1-2\gamma})}\\
=&\int_{\partial'B^+_r(z_0)}|\tilde{\eta}_\alpha\tilde{u}_\alpha|^{2^*-2}
(\tilde{\eta}_\alpha\tilde{u}_\alpha)\psi \,d\sigma_{\tilde{h}_\alpha}
+o(1)\|\hat{\psi}_\alpha\|_{W^{1,2}(X,\rho^{1-2\gamma})}.
\end{split}\end{equation}
 Since $\|\psi\|_{W^{1,2}(B^+_r(z_0),y^{1-2\gamma})}=\|\hat{\psi}_\alpha\|_{W^{1,2}(X,\rho^{1-2\gamma})}$, combining expressions \eqref{equation21} and \eqref{equation22} then we have
\begin{equation*}\begin{split}
o(1)\|\psi\|_{W^{1,2}(B^+_r(z_0),y^{1-2\gamma})}
=&\int_{\partial'B^+_r(z_0)}
\lim_{y\rightarrow0}y^{1-2\gamma}\partial_y(\tilde{\eta}_\alpha\tilde{u}_\alpha)\psi \,d\sigma_{\tilde{h}_\alpha}\\
&+\int_{\partial'B^+_r(z_0)}|\tilde{\eta}_\alpha\tilde{u}_\alpha|^{2^*-2}
(\tilde{\eta}_\alpha\tilde{u}_\alpha)\psi \,d\sigma_{\tilde{h}_\alpha},
\end{split}\end{equation*}
i.e.
$$
-\lim_{y\rightarrow0}y^{1-2\gamma}\partial_y(\tilde{\eta}_\alpha\tilde{u}_\alpha)
=|\tilde{\eta}_\alpha\tilde{u}_\alpha|^{2^*-2}(\tilde{\eta}_\alpha\tilde{u}_\alpha)+o(1)
\ \ \hbox{on} \ \ \partial'B^+_r(z_0).
$$
Meanwhile, since $\tilde{\eta}_\alpha \tilde{u}_\alpha\rightharpoonup u$ weakly in $W^{1,2}(\mathbb{R}^{n+1}_+,y^{1-2\gamma})$, the same argument as above gives that
$$
-\lim_{y\rightarrow 0}y^{1-2\gamma}\partial_yu=|u|^{2^*-2}u
\ \ \hbox{on} \ \ \partial'B^+_r(z_0).
$$
If we denote by
$$
\Gamma_\alpha
:=|\tilde{\eta}_\alpha\tilde{u}_\alpha|^{2^*-2}(\tilde{\eta}_\alpha\tilde{u}_\alpha)
-|u|^{2^*-2}u
-|\tilde{\eta}_\alpha\tilde{u}_\alpha-u|^{2^*-2}(\tilde{\eta}_\alpha\tilde{u}_\alpha-u),
$$
then
\begin{equation}
\left\{
\begin{split}
-\divergence(y^{1-2\gamma}\nabla (\tilde{\eta}_\alpha\tilde{u}_\alpha-u))=o(1) &  \quad \hbox{in}  \ \ B^+_r(z_0),\\
-\lim_{y\rightarrow0}y^{1-2\gamma}\partial_y(\tilde{\eta}_\alpha\tilde{u}_\alpha-u)
=|\tilde{\eta}_\alpha\tilde{u}_\alpha-u|^{2^*-2}(\tilde{\eta}_\alpha\tilde{u}_\alpha-u)+\Gamma_\alpha+o(1)
&\quad   \hbox{on} \ \ \partial'B^+_r(z_0).
\end{split} \right.
\end{equation}
We have proved in \eqref{energy bound} that for any $r>0$ and $\varepsilon_1\in(0,\varepsilon_0)$, there exists a sequence $\{\mu_\alpha\}_{\alpha\in\mathbb{N}}$ such that, if $|z_0|+r< r_0\leq\mu^{-1}_\alpha r_0$, it holds that
$$
\int_{\partial'B^+_r(z_0)}|\tilde{u}_\alpha|^{2^*}dx\leq\frac{\varepsilon_1}{2}.
$$
Therefore we  can also choose small $r\in (0,\frac{r_0}{3})$ and $|z_0|<2r$ such that
$$
\int_{\partial'B^+_r(z_0)}|\tilde{\eta}_\alpha\tilde{u}_\alpha-u|^{2^*}dx\leq\varepsilon_1.
$$
We claim that $\Gamma_\alpha=o(1)$ in the sense that for any $\phi\in W^{1,2}(\mathbb{R}^{n+1}_+,y^{1-2\gamma})'$, we have
$$
\int_{\partial'B^+_r(z_0)}|\Gamma_\alpha \phi| d\sigma_{\hat{h}}=o(1)||\phi||_{L^{2^*}(\partial'B^+_r(z_0))} \ \ \hbox{as}   \  \alpha\to +\infty.
$$
We can use the same arguments as in the proof of Lemma \ref{limit lemma} to show this claim.

Then by Lemma \ref{epsilon regularity lemma} with $\varepsilon=\varepsilon_1$ and Prposition \ref{fractional compact embedding}, we can prove that $\tilde{\eta}_\alpha\tilde{u}_\alpha\rightarrow u$ in $W^{1,2}(B^+_r(z_0),y^{1-2\gamma})$ for $|z_0|<2r$, then by the finite covering we can prove that $\tilde{\eta}_\alpha\tilde{u}_\alpha\rightarrow u$ in $W^{1,2}(B^+_{2r}(0),y^{1-2\gamma})$ for $0<r<r_0/8$.

\end{proof}

Applying Claim 3, noting that $\tilde{\eta}_\alpha\tilde{u}_\alpha\rightarrow u$ in $W^{1,2}(B^+_{2r}(0),y^{1-2\gamma})$, and that $\tilde{\eta}_\alpha\equiv1$ in $D_{1/4\mu^{-1}_\alpha r_0}$, since $0<\mu_\alpha\leq1$ and $r\in(0,r_0/8)$, we have
\begin{equation*}\begin{split}
\varepsilon&=\int_{D_{2r}(0)}|\tilde{u}_\alpha|^{2^*}\,d\sigma_{\tilde{h}_\alpha}
=\int_{D_{2r}(0)}|\tilde{\eta}_\alpha\tilde{u}_\alpha|^{2^*}\,d\sigma_{\tilde{h}_\alpha}\\
&\leq2\int_{D_{2r}(0)}|u|^{2^*}\,dx+o(1),
\end{split}\end{equation*}
where we used $\tilde{\eta}_\alpha\tilde{u}_\alpha\rightarrow u$ in $L^{2^*}(D_{2r}(0),|dx|^2)$ as $\alpha\rightarrow+\infty$ by Proposition \ref{weighted trace sobolev embedding}.
So $u\neq0$.

\begin{claim}
$\lim_{\alpha\rightarrow+\infty}\mu_\alpha=0$.
\end{claim}
In fact, if $\mu_{\alpha}\rightarrow\mu_0>0$, then $\tilde{\eta}_\alpha\tilde{u}_\alpha\rightharpoonup0$ in $W^{1,2}(B^+_{2r}(0),y^{1-2\gamma})$ since $\hat{u}_\alpha\rightharpoonup0$ in $W^{1,2}(X,\rho^{1-2\gamma})$. But $u\neq0$,  which is a contradiction.

\begin{claim}
For any $0<\mu_0\leq1$, $\tilde{u}_\alpha\rightarrow u$ strongly in $W^{1,2}(B^+_{\mu^{-1}_0}(0),y^{1-2\gamma})$ as $\alpha\rightarrow+\infty$, and $u$ is a weak solution of equation (\ref{Liouville equation}).
\end{claim}
\begin{proof}
Let $0<\mu_0\leq 1$, by Claim 4, we know $0<\mu_\alpha\leq\mu_0$ for $\alpha$ large. Then \eqref{energy bound} holds for $|z_0|+r<\mu^{-1}_0r_0$. By the same arguments, it is easy to check that
$$
\tilde{\eta}_\alpha\tilde{u}_\alpha\rightarrow u \ \ \hbox{in} \ \ W^{1,2}(B^+_{2r\mu^{-1}_0}(0),y^{1-2\gamma}).
$$
For $\alpha$ large, we have $\tilde{\eta}_\alpha\equiv1\ \ \hbox{in} \ \ B^+_{2r\mu^{-1}_0}(0)$, so we have
$$
\tilde{u}_\alpha\rightarrow u \ \ \hbox{in} \ \ W^{1,2}(B^+_{2r\mu^{-1}_0}(0),y^{1-2\gamma})
$$
strongly as $\alpha\rightarrow+\infty$.\\

We finally claim that $u$ solves the following boundary problem.
\begin{equation}
\left\{
\begin{split}
-\divergence(y^{1-2\gamma}\nabla u)=0& \ \ \hbox{in} \ \ \mathbb{R}^{n+1}_+,\\
-\lim_{y\rightarrow0}y^{1-2\gamma}\partial_y u=|u|^{2^*-2}u& \ \ \hbox{on} \ \ \mathbb{R}^n.\\
\end{split}
\right.
\end{equation}
Since $0<\mu_0\leq1$ is arbitrary, we have $\tilde{u}_\alpha\to u$ strongly in $W^{1,2}(B^+_R(0),y^{1-2\gamma})$ for any large $R>0$.
Without loss of generality, let $\psi\in \mathcal C^\infty_0(\mathbb{R}^{n+1}_+)$ and $\hbox{supp}\,\psi\subset B^+_0(R_0)$ for some $R_0>0$. Set
$$
\psi_\alpha(z)=\mu^{-\frac{n-2\gamma}{2}}_\alpha\psi(\mu^{-1}_\alpha\varphi^{-1}_{x_\alpha}(z)).
$$
For $\alpha$ large enough, we have
$$
\int_X\rho^{1-2\gamma}\langle\nabla\hat{u}_\alpha,\nabla\psi_\alpha\rangle_gdv_g
=\int_{\mathbb{R}^{n+1}_+}\tilde{\rho}^{1-2\gamma}_\alpha\langle\nabla(\tilde{\eta}_\alpha\tilde{u}_\alpha),\nabla\psi\rangle_{\tilde{g}_\alpha}dv_{\tilde{g}_\alpha},
$$
and
$$
\int_M|\hat{u}_\alpha|^{2^*-2}\hat{u}_\alpha\psi_\alpha dv_g
=\int_{\mathbb{R}^n}|\tilde{\eta}_\alpha\tilde{u}_\alpha|^{2^*-2}
(\tilde{\eta}_\alpha\tilde{u}_\alpha)\psi\, dv_{\tilde{g}_\alpha}.
$$
Note that $\tilde{g}_\alpha\to|dx|^2+dy^2$ in $\mathcal C^1(B^+_R(0))$ as $\alpha\to+\infty$, $\{\hat{u}_\alpha\}$ is a Palais-Smale sequence for $I^\gamma_g$ and $\tilde{\eta}_\alpha\tilde{u}_\alpha\to u$ in $W^{1,2}(B^+_R(0))$ for any $R>0$. Then we have
$$
\int_{\mathbb{R}^{n+1}_+}y^{1-2\gamma}\langle\nabla u,\nabla\psi\rangle \,dxdy
-\int_{\mathbb{R}^n}|u|^{2^*-2}u\psi\, dxdy=0,
$$
which yields our desired result.

\end{proof}

{\bf Step 4.} The Palais-Smale sequence subtracted by a bubble is still a Palais-Smale sequence. Define
\begin{equation}\label{bubble}
    \left\{
  \begin{array}{ll}
  \hat{w}_\alpha(z)=\hat{\eta}_\alpha(z)\mu^{-(n-2\gamma)/2}_\alpha u(\mu^{-1}_\alpha\varphi^{-1}_{x_\alpha}(z)) ,
 & \ z\in\varphi_{x_\alpha}(B^+_{2r_0}(0)),\\
  \hat{w}_\alpha(z)=0, & \ \ \hbox{otherwise},
 \end{array}
\right.
\end{equation}
where  $\hat{\eta}_\alpha$ is a cut-off function satisfying $\hat{\eta}_\alpha=1$ in $\varphi_{x_\alpha}(B^+_{r_0}(0))$ and  $\hat{\eta}_\alpha=0$ in $M\setminus\varphi_{x_\alpha}(B^+_{2r_0}(0))$. Here we have $\mathfrak{B}^+_{2r_0}(x_\alpha)=\varphi_{x_\alpha}(B^+_{2r_0}(0))$. Let $\hat{v}_\alpha=\hat{u}_\alpha-\hat{w}_\alpha$. We claim:
\begin{itemize}
\item[(i)] $\hat{v}_\alpha\rightharpoonup0$ in $W^{1,2}(X,\rho^{1-2\gamma})$ as $\alpha\rightarrow+\infty$;

\item[(ii)] $DI^\gamma_g(\hat{v}_\alpha)\rightarrow0$ in $W^{1,2}(X,\rho^{1-2\gamma})'$ as $\alpha\rightarrow+\infty$;

\item[(iii)] $I^\gamma_g(\hat{v}_\alpha)=I^\gamma_g(\hat{u}_\alpha)-\tilde{E}(u)+o(1)$ as  $\alpha\rightarrow+\infty$;

\item[(iv)] $ \{\hat{v}_\alpha\}_{\alpha\in\mathbb{N}} $ is also a Palais-Smale sequence for $I^\gamma_g$.
\end{itemize}

The proof of these claims follows from:
(i) Since $\hat{u}_\alpha\rightharpoonup0$ in $W^{1,2}(X,\rho^{1-2\gamma})$ as $\alpha\rightarrow+\infty$, it suffices to prove $\hat{w}_\alpha\rightharpoonup0$ in $W^{1,2}(X,\rho^{1-2\gamma})$ as $\alpha\rightarrow+\infty$. First, we prove that $\int_M\hat{w}_\alpha\psi d\sigma_{\hat{h}}=o(1)$ as $\alpha\rightarrow+\infty$ for any $\psi\in \mathcal C^\infty(\overline{X})$. Given $R>0$, then
\begin{equation}\label{equation30}
\int_M\hat{w}_\alpha\psi \,d\sigma_{\hat{h}}=\int_{\mathfrak{D}_{\mu_\alpha R}(x_\alpha)}\hat{w}_\alpha\psi \,d\sigma_{\hat{h}}+\int_{M\setminus \mathfrak{D}_{\mu_\alpha R}(x_\alpha)}\hat{w}_\alpha\psi\, d\sigma_{\hat{h}}.
\end{equation}
Note that $\tilde{h}_\alpha(x)=(\varphi^*_{x_\alpha}\hat{h})(\mu_\alpha x)$. Using \eqref{bubble} we have
\begin{equation*}\begin{split}
\int_{\mathfrak{D}_{\mu_\alpha R}(x_\alpha)}\hat{w}_\alpha\psi \,d\sigma_{\hat{h}}
&=\int_{\mathfrak{D}_{\mu_\alpha R}(x_\alpha)}\hat{\eta}_\alpha(x)\mu^{-\frac{n-2\gamma}{2}}_\alpha u(\mu^{-1}_\alpha\varphi^{-1}_{x_\alpha}(x))\psi(x)\, d\sigma_{\hat{h}}\\
&=\mu_\alpha^{\frac{n+2\gamma}{2}}\int_{D_{R}(0)}\hat{\eta}_\alpha(\varphi_{x_\alpha}(\mu_\alpha x))u(x)\psi(\varphi_{x_\alpha}(\mu_\alpha x))\, d\sigma_{\tilde{h}_\alpha}\\
&\leq C\|\psi\|_{L^\infty(M)} \mu_\alpha^{\frac{n+2\gamma}{2}}\int_{D_{R}(0)}|u(x)|\,dx.
\end{split}\end{equation*}
Similarly, we can deal with the second term in the right hand side of \eqref{equation30}:
\begin{equation*}\begin{split}
&\int_{M\setminus \mathfrak{D}_{\mu_\alpha R}(x_\alpha)}\hat{w}_\alpha\psi\, d\sigma_{\hat{h}}
=\int_{\mathfrak{D}_{2r_0}(x_\alpha)\setminus \mathfrak{D}_{\mu_\alpha R}(x_\alpha)}\hat{w}_\alpha\psi \,d\sigma_{\hat{h}}\\
&\quad\leq C\|\psi\|_{L^\infty(M)} \mu_\alpha^{\frac{n+2\gamma}{2}}\int_{D_{2r_0\mu^{-1}_\alpha}(0)\setminus D_R(0)}|u(x)|\,dx\\
&\quad\leq C\|\psi\|_{L^\infty(M)} \mu_\alpha^{\frac{n+2\gamma}{2}}
\left(\int_{D_{2r_0\mu^{-1}_\alpha}(0)\setminus D_R(0)}|u(x)|^{2^*}dx\right)^{\frac{1}{2^*}}
\left(\int_{D_{2r_0\mu^{-1}_\alpha}(0)\setminus D_R(0)}dx\right)^{\frac{n+2\gamma}{2n}}\\
&\quad\leq C\|\psi\|_{L^\infty(M)}\left(\int_{D_{2r_0\mu^{-1}_\alpha}(0)\setminus D_R(0)}|u(x)|^{2^*}\,dx\right)^{\frac{1}{2^*}}.
\end{split}\end{equation*}
Since $u\in L^{2^*}(\mathbb{R}^n,|dx|^2)$ and $\mu_\alpha\rightarrow0$ as $\alpha\rightarrow+\infty$, taking $R$ large enough we get $\int_M\hat{w}_\alpha\psi d\sigma_{\hat{h}}=o(1)$ as $\alpha\rightarrow+\infty$.\\

 Next, we will show that $\int_X\rho^{1-2\gamma}\langle\nabla\hat{w}_\alpha,\nabla\psi\rangle_g dv_g=o(1)$ as $\alpha\rightarrow+\infty$  for any $\psi\in \mathcal C^\infty(\overline{X})$. Let $\tilde{\eta}_\alpha(z)=\hat{\eta}_\alpha(\varphi_{x_\alpha}(\mu_\alpha z))$, $\tilde{\rho}_\alpha(z)=\mu^{-1}_\alpha\rho(\varphi_{x_\alpha}(\mu_\alpha z))$.  Noting that $\hat{w}_\alpha\equiv0$ in $X\setminus\mathfrak{B}^+_{2r_0}(x_\alpha)$, then for any $R>0$ and $\alpha$ large, we have
\begin{equation}\label{formula32}\begin{split}
\int_X\rho^{1-2\gamma}&\langle\nabla \hat{w}_\alpha,\nabla \psi\rangle_g \,dv_g
=\int_{\mathfrak{B}^+_{2r_0}(x_\alpha)}\rho^{1-2\gamma}\langle\nabla \hat{w}_\alpha,\nabla \psi\rangle_g \,dv_g\\
=&\int_{\mathfrak{B}^+_{2r_0}(x_\alpha)\setminus \mathfrak{B}^+_{R\mu_\alpha}(x_\alpha)}
\rho^{1-2\gamma}\langle\nabla \hat{w}_\alpha,\nabla \psi\rangle_g \,dv_g
+\int_{\mathfrak{B}^+_{R\mu_\alpha}(x_\alpha)}
\rho^{1-2\gamma}\langle\nabla \hat{w}_\alpha,\nabla \psi\rangle_g \,dv_g\\
=&:I_1+I_2.
\end{split}\end{equation}
By H\"{o}lder's inequality and that $u\in W^{1,2}(\mathbb{R}^{n+1}_+,y^{1-2\gamma})$, we have
\begin{equation*}\begin{split}
I_1&\leq\left(\int_{\mathfrak{B}^+_{2r_0}(x_\alpha)\setminus \mathfrak{B}^+_{R\mu_\alpha}(x_\alpha)}
\rho^{1-2\gamma}|\nabla \hat{w}_\alpha|^2_g\, dv_g\right)^{\frac{1}{2}}
\left(\int_{\mathfrak{B}^+_{2r_0}(x_\alpha)\setminus \mathfrak{B}^+_{R\mu_\alpha}(x_\alpha)}
\rho^{1-2\gamma}|\nabla \psi|^2_g \,dv_g\right)^{\frac{1}{2}}\\
&=\left(\int_{B^+_{2r_0\mu^{-1}_\alpha}(0)\setminus B^+_R(0)}
\tilde{\rho}^{1-2\gamma}_\alpha
|\nabla (\tilde{\eta}_\alpha u)|^2_{\tilde{g}_\alpha}\,dv_{\tilde{g}_\alpha}\right)^{\frac{1}{2}}
\left(\int_{\mathfrak{B}^+_{2r_0}(x_\alpha)\setminus \mathfrak{B}^+_{R\mu_\alpha}(x_\alpha)}
\rho^{1-2\gamma}|\nabla \psi|^2_g \,dv_g\right)^{\frac{1}{2}}\\
&=:\beta_\alpha(R),
\end{split}\end{equation*}
where
\begin{equation}\label{beta}
\lim_{R\rightarrow+\infty}\lim_{\alpha\rightarrow+\infty}\ \sup \beta_\alpha(R)=0.
\end{equation}
The previous limit is estimated because $u\in W^{1,2}(\mathbb{R}^{n+1}_+,y^{1-2\gamma})$, so we have for any $\alpha,R$
$$
\left(\int_{B^+_{2r_0\mu^{-1}_\alpha}(0)\setminus B^+_R(0)}
\tilde{\rho}^{1-2\gamma}_\alpha
|\nabla (\tilde{\eta}_\alpha u)|^2_{\tilde{g}_\alpha}\,dv_{\tilde{g}_\alpha}\right)^{\frac{1}{2}}
\leq C||u||_{W^{1,2}(\mathbb{R}^{n+1}_+,y^{1-2\gamma})},
$$
and  for any $\varepsilon>0$ and any $\alpha$ large, there exists $R_0>0$ such that for $R>R_0$, we have
$$
\left(\int_{\mathfrak{B}^+_{2r_0}(x_\alpha)\setminus \mathfrak{B}^+_{R\mu_\alpha}(x_\alpha)}
\rho^{1-2\gamma}|\nabla \psi|^2_g \,dv_g\right)^{\frac{1}{2}}
\leq \varepsilon.
$$

Meanwhile we have
\begin{equation*}\begin{split}
I_2&\leq\left(\int_{\mathfrak{B}^+_{R\mu_\alpha}(x_\alpha)}
\rho^{1-2\gamma}|\nabla \hat{w}_\alpha|^2_g \,dv_g\right)^{\frac{1}{2}}
\left(\int_{\mathfrak{B}^+_{R\mu_\alpha}(x_\alpha)}
\rho^{1-2\gamma}|\nabla \psi|^2_g\, dv_g\right)^{\frac{1}{2}}\\
&=\left(\int_{B^+_R(0)}\tilde{\rho}^{1-2\gamma}_\alpha
|\nabla (\tilde{\eta}_\alpha u)|^2_{\tilde{g}_\alpha}\,dv_{\tilde{g}_\alpha}\right)^{\frac{1}{2}}
\left(\int_{\mathfrak{B}^+_{R\mu_\alpha}(x_\alpha)}
\rho^{1-2\gamma}|\nabla \psi|^2_g \,dv_g\right)^{\frac{1}{2}}\\
&=o(1),
\end{split}\end{equation*}
uniformly in $R$ as $\alpha\rightarrow+\infty$. To see this, for any $R>0$,
$$
\left(\int_{B^+_R(0)}\tilde{\rho}^{1-2\gamma}_\alpha
|\nabla (\tilde{\eta}_\alpha u)|^2_{\tilde{g}_\alpha}\,dv_{\tilde{g}_\alpha}\right)^{\frac{1}{2}}
\leq C||u||_{W^{1,2}(\mathbb{R}^{n+1}_+,y^{1-2\gamma})},
$$
also in Claim 4 we have proved that
$$
\lim_{\alpha\to +\infty}\mu_\alpha=0
$$
and note that $\psi\in W^{1,2}(X,\rho^{1-2\gamma})$.
Since $R>0$ is arbitrary, \eqref{formula32} implies that
$$
\int_X\rho^{1-2\gamma}\langle\nabla \hat{w}_\alpha,\nabla \psi\rangle_g\, dv_g=o(1)
$$
as $\alpha\rightarrow+\infty$.\\

(ii) For any $\psi\in W^{1,2}(X,\rho^{1-2\gamma})$, the proof of (i), and Propositions \ref{weighted trace sobolev embedding} and \ref{fractional compact embedding} imply that
$$
DI^\gamma_g(\hat{w}_\alpha)\cdot\psi
=\int_X\rho^{1-2\gamma}\langle\nabla\hat{w}_\alpha,\nabla\psi\rangle_g\,dv_g
-\int_M|\hat{w}_\alpha|^{2^*-2}\hat{w}_\alpha\psi \,d\sigma_{\hat{h}}\rightarrow0, \ \ \hbox{as} \ \ \alpha\rightarrow+\infty.
$$
On the other hand, we have
\begin{equation*}\begin{split}
DI^\gamma_g(\hat{v}_\alpha)\cdot\psi
&=\int_X\rho^{1-2\gamma}\langle\nabla \hat{v}_\alpha,\nabla\psi\rangle_g\,dv_g-\int_M|\hat{v}_\alpha|^{2^*-2}\hat{v}_\alpha\psi \,d\sigma_{\hat{h}}\\
&=DI^\gamma_g(\hat{u}_\alpha)\cdot\psi-DI^\gamma_g(\hat{w}_\alpha)\cdot\psi-\int_M\Phi_\alpha\psi \,d\sigma_{\hat{h}},
\end{split}\end{equation*}
where
$$
\Phi_\alpha=|\hat{u}_\alpha-\hat{w}_\alpha|^{2^*-2}
(\hat{u}_\alpha-\hat{w}_\alpha)+|\hat{w}_\alpha|^{2^*-2}\hat{w}_\alpha
-|\hat{u}_\alpha|^{2^*-2}\hat{u}_\alpha.
$$
Following the same argument of \cite{D-H-R} (pp. 39-40), we can prove that
$$
\int_M\Phi_\alpha\psi \,d\sigma_{\hat{h}}\rightarrow0 \ \ \ \hbox{as} \ \alpha\rightarrow+\infty.
$$
Then we get that $DI^\gamma_g(\hat{v}_\alpha)\rightarrow0$ in $W^{1,2}(X,\rho^{1-2\gamma})'$ as $\alpha\rightarrow+\infty$, since $\{\hat{u}_\alpha\}_{\alpha\in\mathbb{N}}$ is a Palais-Smale sequence for $I^\gamma_g$.\\

(iii)  Note that $\hat{v}_\alpha=\hat{u}_\alpha-\hat{w}_\alpha$ and $\hat{w}_\alpha\equiv0$ in $X\setminus \mathfrak{B}^+_{2r_0}(x_\alpha)$. Given $R>0$, for $\alpha$ large, we have
\begin{equation}\label{equation33}\begin{split}
\int_X\rho^{1-2\gamma}&|\nabla\hat{v}_\alpha|^2_g\,dv_g\\
=&\int_{\mathfrak{B}^+_{2r_0}(x_\alpha)}\rho^{1-2\gamma}|\nabla\hat{v}_\alpha|^2_g\,dv_g
+\int_{X\setminus \mathfrak{B}^+_{2r_0}(x_\alpha)}\rho^{1-2\gamma}|\nabla\hat{u}_\alpha|^2_g\,dv_g\\
=&\int_{\mathfrak{B}^+_{\mu_\alpha R}(x_\alpha)}\rho^{1-2\gamma}|\nabla\hat{v}_\alpha|^2_g\,dv_g
+\int_{\mathfrak{B}^+_{2r_0}(x_\alpha)\setminus \mathfrak{B}^+_{\mu_\alpha R}(x_\alpha)}\rho^{1-2\gamma}|\nabla\hat{v}_\alpha|^2_g\,dv_g\\
&+\int_{X\setminus \mathfrak{B}^+_{2r_0}(x_\alpha)}\rho^{1-2\gamma}|\nabla\hat{u}_\alpha|^2_g\,dv_g\\
=&:I_1+I_2+\int_{X\setminus \mathfrak{B}^+_{2r_0}(x_\alpha)}\rho^{1-2\gamma}|\nabla\hat{u}_\alpha|^2_g\,dv_g.
\end{split}\end{equation}
Since $\tilde{\eta}_\alpha\tilde{u}_\alpha\rightarrow u$ in $W^{1,2}(\mathbb{R}^{n+1}_+,y^{1-2\gamma})$ as $\alpha\rightarrow+\infty$ because of Claim 5, then
\begin{equation*}\begin{split}
I_1=&\int_{\mathfrak{B}^+_{\mu_\alpha R}(x_\alpha)}\rho^{1-2\gamma}|\nabla(\hat{u}_\alpha-\hat{w}_\alpha)|^2_g\,dv_g
=\int_{B^+_R(0)}\tilde{\rho}^{1-2\gamma}_\alpha
|\nabla(\tilde{u}_\alpha-u)|^2_{\tilde{g}_\alpha}\,dv_{\tilde{g}_\alpha}\\
\leq &\,2\int_{B^+_R(0)}y^{1-2\gamma}|\nabla(\tilde{u}_\alpha-u)|^2\,dxdy
=o(1), \ \ \ \hbox{as} \ \alpha\rightarrow+\infty,
\end{split}\end{equation*}
where we have used that $\tilde{\eta}_\alpha\equiv1$ in $B^+_R(0)$ for $\alpha$ large.

On the other hand, direct computations give that
\begin{equation*}\begin{split}
\int_{\mathfrak{B}^+_{2r_0}(x_\alpha)\setminus\mathfrak{B}^+_{\mu_\alpha R}(x_\alpha)}&\rho^{1-2\gamma}|\nabla\hat{w}_\alpha|^2_g\,dv_g
=\int_{B^+_{2r_0\mu_\alpha^{-1}}(0)\setminus B^+_R(0)}\tilde{\rho}^{1-2\gamma}_\alpha|\nabla u|^2_{\tilde{g}_\alpha}\,dv_{\tilde{g}_\alpha}\\
\leq& 2\int_{B^+_{2r_0\mu_\alpha^{-1}}(0)\setminus B^+_R(0)}y^{1-2\gamma}|\nabla u|^2\,dxdy=\beta_\alpha(R),
\end{split}\end{equation*}
since $u\in W^{1,2}(\mathbb{R}^{n+1}_+,y^{1-2\gamma})$ and $\mu_\alpha\rightarrow0$ as $\alpha\rightarrow+\infty$, where $\beta_\alpha(R)$ is defined as in \eqref{beta}. Hence we get that
\begin{equation*}\begin{split}
I_2=&\int_{\mathfrak{B}^+_{2r_0}(x_\alpha)\setminus\mathfrak{B}^+_{\mu_\alpha R}(x_\alpha)}\rho^{1-2\gamma}
(|\nabla\hat{u}_\alpha|^2_g+|\nabla\hat{w}_\alpha|^2_g-2\langle\nabla\hat{u}_\alpha,
\nabla\hat{w}_\alpha\rangle_g )\,dv_g\\
=&\int_{\mathfrak{B}^+_{2r_0}(x_\alpha)\setminus\mathfrak{B}^+_{\mu_\alpha R}(x_\alpha)}\rho^{1-2\gamma}|\nabla\hat{u}_\alpha|^2_g\,dv_g+\beta_\alpha(R).
\end{split}\end{equation*}
Here we have used H\"{o}lder's inequality and the fact that $\{\hat{u}_\alpha\}$ is uniformly in $W^{1,2}(X,\rho^{1-2\gamma})$ to get
$$
\int_{\mathfrak{B}^+_{2r_0}(x_\alpha)\setminus\mathfrak{B}^+_{\mu_\alpha R}(x_\alpha)}\rho^{1-2\gamma}
\langle\nabla\hat{u}_\alpha,\nabla\hat{w}_\alpha\rangle_g \,dv_g=\beta_\alpha(R).
$$
Therefore, noting that $\tilde{u}_\alpha\rightarrow u$ in $W^{1,2}(\mathbb{R}^{n+1}_+,y^{1-2\gamma})$ as $\alpha\rightarrow+\infty$, we have from \eqref{equation33} that
\begin{equation*}\begin{split}
\int_X\rho^{1-2\gamma}&|\nabla\hat{v}_\alpha|^2_g\,dv_g\\
=&\int_X\rho^{1-2\gamma}|\nabla\hat{u}_\alpha|^2_g\,dv_g
-\int_{\mathfrak{B}^+_{\mu_\alpha R}(x_\alpha)}\rho^{1-2\gamma}|\nabla\hat{u}_\alpha|^2_g\,dv_g+\beta_\alpha(R)+o(1)\\
=&\int_X\rho^{1-2\gamma}|\nabla\hat{u}_\alpha|^2_g\,dv_g
-\int_{B^+_R(0)}\tilde{\rho}^{1-2\gamma}_\alpha|\nabla\tilde{u}_\alpha|^2_{\tilde{g}_\alpha}
\,dv_{\tilde{g}_\alpha}+\beta_\alpha(R)+o(1)\\
=&\int_X\rho^{1-2\gamma}|\nabla\hat{u}_\alpha|^2_g\,dv_g
-\int_{B^+_R(0)}y^{1-2\gamma}|\nabla u|^2\,dxdy+\beta_\alpha(R)+o(1)\\
=&\int_X\rho^{1-2\gamma}|\nabla\hat{u}_\alpha|^2_g\,dv_g
-\int_{\mathbb{R}^{n+1}_+}y^{1-2\gamma}|\nabla u|^2\,dxdy+\beta_\alpha(R)+o(1).
\end{split}\end{equation*}
In a similar way, we can get that
$$
\int_M|\hat{v}_\alpha|^{2^*}\,d\sigma_{\hat{h}}
=\int_M|\hat{u}_\alpha|^{2^*}d\sigma_{\hat{h}}-\int_{\mathbb{R}^n}|u|^{2^*}\,dx+\beta_\alpha(R)+o(1).
$$
These imply that
$$
I^\gamma_g(\hat{v}_\alpha)=I^\gamma_g(\hat{u}_\alpha)-\tilde{E}(u)+\beta_\alpha(R)+o(1).
$$
Since $R>0$ is arbitrary, we get conclusion (iii).\\

(iv) It is a direct consequence of (ii) and (iii).

\end{proof}


\section{Proof of the main Results}

{\bf Proof of Theorem \ref{main theorem}}. From Remark \ref{remark on limit function}, we have $u_\alpha\rightharpoonup u^0$ in $ W^{1,2}(X,\rho^{1-2\gamma})$ as $\alpha\rightarrow+\infty$. And $u_\alpha\rightarrow u^0$ a.e. on $M$ as $\alpha\rightarrow+\infty$.  Then $u^0\geq0$ on $M$ since $u_\alpha\geq0$. Also $\hat{u}_\alpha=u_\alpha-u^0$ satisfies the Palais-Smale condition and
$$
I^\gamma_g(\hat{u}_\alpha)=I^{\gamma,\alpha}_g(u_\alpha)-I^{\gamma,\infty}_g(u^0)+o(1).
$$
If $\hat{u}_\alpha\rightarrow0$ in $W^{1,2}(X,\rho^{1-2\gamma})$ as $\alpha\rightarrow+\infty$, then the theorem is proved. If $\hat{u}_\alpha\rightharpoonup0$ but not strongly in $W^{1,2}(X,\rho^{1-2\gamma})$ as $\alpha\rightarrow+\infty$, using Lemma \ref{first bubble lemma}, we can obtain a new Palais-Smale sequence $\{\hat{u}^1_\alpha\}_{\alpha\in\mathbb{N}}$ satisfying
$$
I^\gamma_g(\hat{u}^1_\alpha)= I^\gamma_g(\hat{u}_\alpha)-\tilde{E}(u)+o(1).
$$
Now again, either $\hat{u}^1_\alpha\rightarrow0$ in $W^{1,2}(X,\rho^{1-2\gamma})$ as $\alpha\rightarrow+\infty$, in which case the theorem holds, or  $\hat{u}^1_\alpha\rightharpoonup0$ but not strongly in $W^{1,2}(X,\rho^{1-2\gamma})$ as $\alpha\rightarrow+\infty$, in which case we again use Lemma \ref{first bubble lemma}. Since $\{I^{\gamma,\alpha}_g(u_\alpha)\}_{\alpha\in\mathbb{N}}$ is uniformly bounded, after a finite number of induction steps, we get the last Palais-Smale sequence $\{\hat{u}^m_\alpha\}_{\alpha\in\mathbb{N}}$ $(m>1)$ with $I^\gamma_g(\hat{u}^m_\alpha)\rightarrow \beta<\beta_0$. Then by Lemma \ref{criterion lemma}, we can get that $\hat{u}^m_\alpha\rightarrow0$ in $W^{1,2}(X,\rho^{2\gamma-1})$ as $\alpha\rightarrow+\infty$. Applying Lemma \ref{first bubble lemma} in the process, we can get $\{u^j\}^m_{j=1}$ are solutions to \eqref{Liouville equation}. We will prove the positivity of $u^j$, $j=1,\cdots,m$, in Lemma \ref{nonnegative lemma}, and the relation (5) of Theorem \ref{main theorem} in Lemma \ref{interfering lemma}.

For the regularity of $u^j$ we can use Lemma \ref{weighted regularity1} and \ref{weighted regularity2} in the Appendix. Then the proof of the theorem is finished.

\begin{lemma}\label{interfering lemma}
For any integer $k$ in $[1,m]$, and any integer $l$ in $[0,k-1]$, there exist an integer $s$ and sequences $\{y^j_\alpha\}_{\alpha\in \mathbb{N}}\subset M$ and $\{\lambda^j_\alpha>0\}_{\alpha\in \mathbb{N}}$, $j=1,\cdots,s$, such that $d_{\hat{h}}(x^k_\alpha, y^j_\alpha)/\mu^k_\alpha$ is bounded and $\lambda^j_\alpha/\mu^k_\alpha\rightarrow0$ as $\alpha\rightarrow+\infty$, and for any $R,R'>0$,
\begin{align}\label{*}
\int_{\mathfrak{D}_{R\mu^k_\alpha}(x^k_\alpha)\setminus
\cup^s_{j=1}\mathfrak{D}_{R'\lambda^j_\alpha}(y^j_\alpha)}
|\hat{u}_\alpha-\sum^l_{i=1}u^i_\alpha-u^k_\alpha|^{2^*}d\sigma_{\hat{h}}
=o(1)+\epsilon_\alpha(R'),
\end{align}
where
$$
\lim_{R'\to+\infty}\lim_{\alpha\to+\infty}\sup\epsilon_\alpha(R')=0,
$$
and $\{u^i_\alpha\}$ is derived from the rescaling of $u^i$ we obtained in the above proof of Theorem \ref{main theorem}, and $\{x^i_\alpha\}$ is the $i$-th likely blow up points sequence.
\end{lemma}

\begin{proof}
 We prove this lemma by iteration on $l$. For any integer $k$ ($1\leq k\leq m$), if $l=k-1$, combining the above proof of Theorem \ref{main theorem} with Lemma \ref{first bubble lemma} and Proposition \ref{weighted trace sobolev embedding}, we have
$$
\int_{\mathfrak{D}_{R\mu^k_\alpha}(x^k_\alpha)}
|\hat{u}_\alpha-\sum^{k-1}_{i=1}u^i_\alpha-u^k_\alpha|^{2^*}\,d\sigma_{\hat{h}}=o(1),
$$
so \eqref{*} holds for $s=0$.\\

Suppose that \eqref{*} holds for some $l$, $1\leq l\leq k-1$, we need to show that \eqref{*} holds for $l-1$.\\

\underline{Case 1} $d_{\hat{h}}(x^l_\alpha,x^k_\alpha)\nrightarrow0$ as $\alpha\rightarrow+\infty$. Then for any $\bar{R}>0$, up to a subsequence, $\mathfrak{D}_{\bar{R}\mu^l_\alpha}(x^l_\alpha)\cap \mathfrak{D}_{R\mu^k_\alpha}(x^k_\alpha)=\emptyset$, so we have
\begin{equation*}\begin{split}
\int_{\mathfrak{D}_{R\mu^k_\alpha}(x^k_\alpha)
\setminus\cup^s_{j=1}\mathfrak{D}_{R'\lambda^j_\alpha}(y^j_\alpha)}|u^l_\alpha|^{2^*}\,d\sigma_{\hat{h}}
&\leq\int_{M\setminus \mathfrak{D}_{\bar{R}\mu^l_\alpha}(x^l_\alpha)}|u^l_\alpha|^{2^*}\,d\sigma_{\hat{h}}\\
&\leq C\int_{\mathbb{R}^n\setminus D_{\bar{R}}(0)}|u^l|^{2^*}\,d\sigma_{\tilde{h}_\alpha}
\leq C\int_{\mathbb{R}^n\setminus D_{\bar{R}}(0)}|u^l|^{2^*}\,dx.
\end{split}\end{equation*}
Since $\bar{R}>0$ is arbitrary and $u^l\in L^{2^*}(\mathbb{R}^n)$, we get
\begin{equation}
\label{equation40}\int_{\mathfrak{D}_{R\mu^k_\alpha}(x^k_\alpha)\setminus
\cup^s_{j=1}\mathfrak{D}_{R'\lambda^j_\alpha}(y^j_\alpha)}|u^l_\alpha|^{2^*}\,d\sigma_{\hat{h}}=o(1),
 \ \  \ \hbox{as} \ \alpha\rightarrow+\infty.
\end{equation}
So by the induction hypothesis for $l$ and \eqref{equation40} we obtain
\begin{equation*}\begin{split}
&\int_{\mathfrak{D}_{R\mu^k_\alpha}(x^k_\alpha)\setminus
\cup^s_{j=1}\mathfrak{D}_{R'\lambda^j_\alpha}(y^j_\alpha)}
|\hat{u}_\alpha-\sum^{l-1}_{i=1}u^i_\alpha-u^k_\alpha|^{2^*}\,d\sigma_{\hat{h}}\\
&\qquad \leq 2^{2^*-1}\int_{\mathfrak{D}_{R\mu^k_\alpha}(x^k_\alpha)
\setminus\cup^s_{j=1}\mathfrak{D}_{R'\lambda^j_\alpha}(y^j_\alpha)}
|\hat{u}_\alpha-\sum^l_{i=1}u^i_\alpha-u^k_\alpha|^{2^*}\,d\sigma_{\hat{h}}\\
& \qquad \quad+2^{2^*-1}\int_{\mathfrak{D}_{R\mu^k_\alpha}(x^k_\alpha)\setminus
\cup^s_{j=1}\mathfrak{D}_{R'\lambda^j_\alpha}(y^j_\alpha)}
|u^l_\alpha|^{2^*}\,d\sigma_{\hat{h}}\\
&\qquad = o(1)+\epsilon_\alpha(R').
\end{split}\end{equation*}
Thus we have proven that \eqref{*} holds for $l-1$.\\

\underline{Case 2} $d_{\hat{h}}(x^l_\alpha,x^k_\alpha)\rightarrow0$ as $\alpha\rightarrow+\infty$. Let $r_0$ be sufficiently small such that
for any $P\in M$, $x,y\in\mathbb{R}^n$ and $|x|,|y|\leq r_0$,
$$
1/2|x-y|\leq d_{\hat{h}}(\varphi_P(x),\varphi_P(y))\leq 2|x-y|.
$$
Let $\tilde{x}^l_\alpha=(\mu^k_\alpha)^{-1}\varphi^{-1}_{x^k_\alpha}(x^l_\alpha)$, $\tilde{y}^j_\alpha=(\mu^k_\alpha)^{-1}\varphi^{-1}_{x^k_\alpha}(y^j_\alpha)$, then

\begin{equation}\label{balls}\begin{split}
&D_{\frac{R}{2}\frac{\mu^l_\alpha}{\mu^k_\alpha}}\left(\tilde{x}^l_\alpha\right)
\subset (\mu^k_\alpha)^{-1}\varphi^{-1}_{x^k_\alpha}(\mathfrak{D}_{R\mu^l_\alpha}(x^l_\alpha))
\subset D_{2R\frac{\mu^l_\alpha}{\mu^k_\alpha}}\left(\tilde{x}^l_\alpha\right),\\
&D_{\frac{R}{2}\frac{\lambda^j_\alpha}{\mu^k_\alpha}}\left(\tilde{y}^j_\alpha\right)
\subset (\mu^k_\alpha)^{-1}\varphi^{-1}_{x^k_\alpha}(\mathfrak{D}_{R\lambda^j_\alpha}(y^j_\alpha))
\subset D_{2R\frac{\lambda^j_\alpha}{\mu^k_\alpha}}\left(\tilde{y}^j_\alpha\right).
\end{split}\end{equation}
Given $\tilde{R}>0$, from Lemma \ref{first bubble lemma}, Proposition \ref{weighted trace sobolev embedding} and proof of Theorem \ref{main theorem} we have
\begin{equation}\label{equation41}
\int_{\mathfrak{D}_{\tilde{R}\mu^l_\alpha}(x^l_\alpha)}|\hat{u}_\alpha
-\sum^l_{i=1}u^i_\alpha|^{2^*}\,d\sigma_{\hat{h}}=o(1).
\end{equation}
By the assumption for $1\leq l\leq k-1$, i.e.
$$
\int_{\mathfrak{D}_{R\mu^k_\alpha}(x^k_\alpha)\setminus\cup^s_{j=1}
\mathfrak{D}_{R'\lambda^j_\alpha}(y^j_\alpha)}
|\hat{u}_\alpha-\sum^l_{i=1}u^i_\alpha-u^k_\alpha|^{2^*}\,d\sigma_{\hat{h}}
=o(1)+\epsilon_\alpha(R'),
$$
combined with \eqref{equation41} then we get that
$$
\int_{[\mathfrak{D}_{R\mu^k_\alpha}(x^k_\alpha)\setminus
\cup^s_{j=1}\mathfrak{D}_{R'\lambda^j_\alpha}(y^j_\alpha)]\cap \mathfrak{D}_{\tilde{R}\mu^l_\alpha}(x^l_\alpha)}
|u^k_\alpha|^{2^*}\,d\sigma_{\hat{h}}
=o(1)+\epsilon_\alpha(R'),
$$
so using \eqref{balls} we arrive at
\begin{align}\label{**}
\int_{[D_R(0)\setminus\cup^s_{j=1}D_{2R'\lambda^j_\alpha/\mu^k_\alpha}(\tilde{y}^j_\alpha)]
\cap D_{1/2\tilde{R}\mu^l_\alpha/\mu^k_\alpha}(\tilde{x}^l_\alpha)}|u^k|^{2^*}\,d\sigma_{\tilde{h}_\alpha}
=o(1)+\epsilon_\alpha(R').
\end{align}

Next, we consider two scenarios: first, assume $d_{\hat{h}}(x^l_\alpha,x^k_\alpha)/\mu^k_\alpha\rightarrow+\infty$ as $\alpha\rightarrow+\infty$. We claim that
 $d_{\hat{h}}(x^l_\alpha,x^k_\alpha)/\mu^l_\alpha\rightarrow+\infty$ as $\alpha\rightarrow+\infty$. If not, then \eqref{**} with $\tilde{R}$ large enough yields that $\mu^l_\alpha/\mu^k_\alpha\rightarrow0$ as $\alpha\rightarrow+\infty$. Moreover,
 $$
 \frac{d_{\hat{h}}(x^l_\alpha,x^k_\alpha)}{\mu^l_\alpha}=\frac{d_{\hat{h}}(x^l_\alpha,x^k_\alpha)}{\mu^k_\alpha}\frac{\mu^k_\alpha}{\mu^l_\alpha},
 $$
so we can choose $\tilde{R}>0$ such that $\mathfrak{D}_{\tilde{R}\mu^k_\alpha}(x^k_\alpha)\cap \mathfrak{D}_{\tilde{R}\mu^l_\alpha}(x^l_\alpha)=\emptyset$, which reduces to the previous case 1 and, as a consequence, \eqref{*} holds for $l-1$.

Second, if $d_{\hat{h}}(x^l_\alpha,x^k_\alpha)/\mu^k_\alpha\nrightarrow+\infty$ as $\alpha\rightarrow+\infty$, then up to a subsequence,
$d_{\hat{h}}(x^l_\alpha,x^k_\alpha)/\mu^k_\alpha$ converges. Then  (\ref{**}) implies that $\mu^l_\alpha/\mu^k_\alpha\rightarrow+\infty$. Set $y^{s+1}_\alpha=x^l_\alpha$ and $\lambda^{s+1}_\alpha=\mu^l_\alpha$, then
$$
\int_{\mathfrak{D}_{R\mu^k_\alpha}(x^k_\alpha)\setminus\cup^{s+1}_{j=1}
\mathfrak{D}_{R'\lambda^j_\alpha}(y^j_\alpha)}
|\hat{u}_\alpha-\sum^l_{i=1}u^i_\alpha-u^k_\alpha|^{2^*}\,d\sigma_{\hat{h}}
=o(1)+\epsilon_\alpha(R')
$$
and
\begin{equation*}\begin{split}
\int_{\mathfrak{D}_{R\mu^k_\alpha}(x^k_\alpha)\setminus\cup^{s+1}_{j=1}
\mathfrak{D}_{R'\lambda^j_\alpha}(y^j_\alpha)}|u^l_\alpha|^{2^*}\,d\sigma_{\hat{h}}
&\leq \int_{M\setminus \mathfrak{D}_{R'\mu^l_\alpha}(x^l_\alpha)}|u^l_\alpha|^{2^*}\,d\sigma_{\hat{h}}\\
&\leq C\int_{\mathbb{R}^n\setminus D_{R'}(0)}|u^l|^{2^*}\,dx\leq \epsilon_\alpha(R'),
\end{split}\end{equation*}
which yield that
$$
\int_{\mathfrak{D}_{R\mu^k_\alpha}(x^k_\alpha)\setminus\cup^{s+1}_{j=1}
\mathfrak{D}_{R'\lambda^j_\alpha}(y^j_\alpha)}
|\hat{u}_\alpha-\sum^{l-1}_{i=1}u^i_\alpha-u^k_\alpha|^{2^*}\,d\sigma_{\hat{h}}
=o(1)+\epsilon_\alpha(R').
$$
In particular,  \ref{*} holds for $l-1$, as desired. The iteration process is thus completed.

Moreover, we have also shown that  for any $i\neq j$
$$
\frac{\mu^i_\alpha}{\mu^j_\alpha}+\frac{\mu^j_\alpha}{\mu^i_\alpha}
+\frac{d_{\hat{h}}(x^i_\alpha,x^j_\alpha)^2}{\mu^i_\alpha\mu^j_\alpha}\rightarrow+\infty
$$
as $\alpha\rightarrow+\infty$ (c.f.  \cite{A},\cite{D-H-R},\cite{S}). Note that this convergence contains two kinds of bubbles: one case is that  $\mu^i_\alpha=O (\mu^j_\alpha)$ when $\alpha\to+\infty$, then the two blow up points are far away from each other. The other case is that $\mu^i_\alpha=o(\mu^j_\alpha)$ or $\mu^j_\alpha=o(\mu^i_\alpha)$ when $\alpha\to+\infty$, then the distance of the two blow up point cannot be determined. Also we get that $\lambda^j_\alpha/\mu^k_\alpha\rightarrow0$ as $\alpha\rightarrow+\infty$.
\end{proof}

\begin{lemma}\label{nonnegative lemma}
The $u^i$ $(i=0,1,\cdots,m)$ we get in the Theorem \ref{main theorem} are all nonnegative.
\end{lemma}
\begin{proof}
First of all, note that  $u^0\geq0$ in $\overline{X}$ by Proposition \ref{u0 property}. So we just need to prove the positivity of $u^i$ for $i\geq 1$. For any $k\in[1,m]$, taking $l=0$ in Lemma \ref{interfering lemma}, we have
\begin{align}\label{strong convergence}
\int_{\mathfrak{D}_{R\mu^k_\alpha}(x^k_\alpha)\setminus\cup^s_{j=1}
\mathfrak{D}_{R'\lambda^j_\alpha}(y^j_\alpha)}
|\hat{u}_\alpha-U^k_\alpha|^{2^*}\,d\sigma_{\hat{h}}
=o(1)+\epsilon_\alpha(R')
\end{align}
where
$$
U^k_\alpha(x)=(\mu^k_\alpha)^{-\frac{n-2\gamma}{2}}u^k((\mu^k_\alpha)^{-1}\varphi^{-1}_{x^k_\alpha}(x)),
\ \ \hbox{for} \ \ x\in \mathfrak{D}_{R\mu^k_\alpha}(x^k_\alpha)
$$
is called a bubble. Since $u_\alpha=\hat{u}_\alpha+u^0$, then for $x\in D_{r_0/\mu^k_\alpha}(0)\subset\mathbb{R}^n$, where the $r_0$ is the same as the one mentioned in Theorem\ref{main theorem}, we have
$$
u^k_\alpha(x)=\tilde{u}^k_\alpha(x)+\tilde{u}^{0,k}_\alpha(x),
$$
where
\begin{equation*}\begin{split}
&u^k_\alpha(x)=(\mu^k_\alpha)^{\frac{n-2\gamma}{2}}u_\alpha(\varphi_{x^k_\alpha}(\mu^k_\alpha x)),\\
&\tilde{u}^k_\alpha(x)=(\mu^k_\alpha)^{\frac{n-2\gamma}{2}}\hat{u}_\alpha(\varphi_{x^k_\alpha}(\mu^k_\alpha x)),\\
&\tilde{u}^{0,k}_\alpha(x)=(\mu^k_\alpha)^{\frac{n-2\gamma}{2}}u^0(\varphi_{x^k_\alpha}(\mu^k_\alpha x)).
\end{split}\end{equation*}
Then \eqref{strong convergence} implies that
\begin{align} \label{***}
\int_{D_R(0)\setminus\cup^s_{j=1}D_{2R'\lambda^j_\alpha/\mu^k_\alpha}(\tilde{y}^j_\alpha)}
|\tilde{u}^k_\alpha-u^k|^{2^*}\,dx
=o(1)+\epsilon_\alpha(R'),
\end{align}
where $\tilde{y}^j_\alpha=(\mu^k_\alpha)^{-1}\varphi^{-1}_{x^k_\alpha}(y^j_\alpha)$.
 Noting that $\{d_{\hat{h}}(x^k_\alpha,y^j_\alpha)/\mu^k_\alpha\}_{\alpha\in\mathbb{N}}$ is uniformly bounded by Lemma \ref{interfering lemma}, therefore $\{\tilde{y}^j_\alpha\}_{\alpha\in\mathbb{N}}$ is bounded and there exists a subsequence, also denoted by  $\{\tilde{y}^j_\alpha\}$, such that $\tilde{y}^j_\alpha\rightarrow\tilde{y}^j$ as $\alpha\rightarrow+\infty$ for $j=1,\ldots,s$. Combining \eqref{***} with $\lambda^j_\alpha/\mu^k_\alpha\rightarrow0$ as $\alpha\rightarrow+\infty$, we get
$$
\tilde{u}^k_\alpha\rightarrow u^k,  \ \ \ \hbox{in} \ L^{2^*}_{loc}(D_R(0)\setminus Y)
$$
as $\alpha\rightarrow+\infty$  for $Y=\{\tilde{y}^j\}^{s}_{j=1}$, so
$$
\tilde{u}^k_\alpha\rightarrow u^k  \ \ \hbox{a.e.}  \ \ \hbox{in} \ \ \mathbb{R}^n,
$$
since $R>0$ is arbitrary.

Also note that
$$
\int_{\mathfrak{D}_{R\mu^k_\alpha}(x^k_\alpha)}|u^0|^{2^*}\,d\sigma_{\hat{h}}
=\int_{D_R(0)}|\tilde{u}^{0,k}_\alpha|^{2^*}\,d\sigma_{\tilde{h}^k_\alpha},
$$
where $\tilde{h}^k_\alpha(x)=(\varphi^*_{x^k_\alpha}\hat{h})(\mu^k_\alpha x)$.
Then $\mu^k_\alpha\rightarrow0$ as $\alpha\rightarrow+\infty$ and $u^0\in L^{2^*}(M,\hat{h})$ yield that
$$
\tilde{u}^{0,k}_\alpha\rightarrow 0,  \ \ \ \hbox{in} \ L^{2^*}(D_R(0),|dx|^2)
$$
as $\alpha\rightarrow+\infty$, so
$$
\tilde{u}^{0,k}_\alpha\rightarrow 0 \ \ \hbox{a.e.} \ \ \hbox{in}  \ \ \mathbb{R}^n
$$
since $R>0$ is arbitrary.

 In particular, we have shown that $u^k_\alpha\rightarrow u^k$ almost everywhere on $\mathbb{R}^n$ as $\alpha\rightarrow+\infty$. Note that $u_\alpha$ is nonnegative by definition, so $u^k_\alpha\geq0$ on $\mathbb{R}^n$. We conclude that $u^k\geq 0$ on $\mathbb{R}^n$.
\end{proof}


\section{Appendix }

By the standard elliptic estimates, we can prove the $\mathcal C^\infty$ estimates from the $L^\infty$ estimates by Harnack inequality. Here we give two important technique lemmas. 

\begin{lemma}\label{weighted regularity1}\cite{G-Q}
Let $R>0$ and $u$ be a weak solution of
\begin{equation}
\left\{
\begin{split}
-\divergence(y^{1-2\gamma}\nabla u)=0& \ \ \hbox{in} \ \ B^+_{2R}(0),\\
-\lim_{y\rightarrow0}y^{1-2\gamma}\partial_yu=f(x)u+g(x)|u|^{2^*-2}u& \ \ \hbox{on} \ \ D_{2R}(0).
\end{split}
\right.
\end{equation}
Here $f$ and $g$ are smooth functions on $D_{2R}(0)$. Assume that $\lambda=\int_{D_{2R}(0)}|u|^{2^*}dx<\infty$. Then for any $p>1$, there exists a constant $C_p=C(p,\lambda)$ such that
$$
\sup_{B^+_R(0)}|u|+\sup_{D_R(0)}|u|\leq C_p\left\{R^{-\frac{n+2-2\gamma}{p}}\|u\|_{L^p(B^+_{2R}(0))}+R^{-\frac{n}{p}}\|u\|_{L^p(D_{2R}(0))}\right\}.
$$
\end{lemma}

\begin{lemma}\label{weighted regularity2}\cite{J-L-X}
Let $a(x),b(x)\in \mathcal C^\alpha(D_2(0))$ for some $0<\alpha\notin\mathbb{N}$ and $u\in W^{1,2}(\partial'B^+_2(0),y^{1-2\gamma})$ be a weak solution of
\begin{equation}
\left\{
\begin{split}
-\divergence(y^{1-2\gamma}\nabla u)=0& \ \ \hbox{in} \ \ B^+_2(0),\\
-\lim_{y\rightarrow0}y^{1-2\gamma}\partial_yu=a(x)u+b(x)& \ \ \hbox{on} \ \ D_2(0).
\end{split}
\right.
\end{equation}
If $2\gamma+\alpha\notin\mathbb{N}$, then $u(\cdot,0)$ is of $\mathcal C^{2\gamma+\alpha}(D_1(0))$, and
$$
\|u(\cdot,0)\|_{C^{2\gamma+\alpha}(D_1(0))}
\leq C(\|u\|_{L^\infty(B^+_2(0))}+\|b\|_{C^\alpha(D_2(0))})
$$
where $C>0$ depends only on $n,\gamma,\alpha$ and $\|a\|_{\mathcal C^\alpha(D_2(0))}$.
\end{lemma}

\bibliographystyle{amsplain}

\end{document}